\documentclass[preprint]{siamart220329}

\usepackage{lipsum}
\usepackage{amsfonts}
\usepackage{graphicx}
\usepackage{epstopdf}
\usepackage{algorithmic}

\usepackage{soul}
\ifpdf
\DeclareGraphicsExtensions{.eps,.pdf,.png,.jpg}
\else
\DeclareGraphicsExtensions{.eps}
\fi

\usepackage{multirow}
\usepackage{caption}
\newsiamremark{remark}{Remark}
\newsiamremark{hypothesis}{Hypothesis}
\crefname{hypothesis}{Hypothesis}{Hypotheses}
\newsiamthm{claim}{Claim}

\newsiamthm{property}{property}
\newsiamremark{assumption}{Assumption}
\newsiamthm{example}{Example}
\headers{Optimization for phase-field simulation of the ESC}{Z. Zhou, W. Huang, W. Jiang and Z. Zhang}

\title{An operator-splitting optimization approach for phase-field simulation of equilibrium shapes of crystals\thanks{Submitted to the editors' DATE.
\funding{The work of Zhen Zhang was partially supported by the NSFC grant (NO.~12071207), 
and the Guangdong Provincial Key Laboratory of Computational Science and Material Design (No.~2019B030301001). The work of Wei Jiang was supported by the National Natural Science Foundation of China Grant (No.~12271414)}}.}

\author{
	Zeyu Zhou\thanks{Department of Mathematics, Southern University of Science and Technology, Shenzhen 518055, China (\email{zhouzy2021@mail.sustech.edu.cn}).}		
	\and Wen Huang\thanks{School of Mathematical Sciences, Xiamen University, Xiamen 361005, China 
		(\email{wen.huang@xmu.edu.cn}).}
	\and Wei Jiang\thanks{School of Mathematics and Statistics, Hubei Key Laboratory of Computational Science, Wuhan University, Wuhan 430072, China
		(\email{jiangwei1007@whu.edu.cn}).}
	\and Zhen Zhang\thanks{Department of Mathematics, Guangdong Provincial Key Laboratory of Computational Science and Material Design, International Center for Mathematics, National Center for Applied Mathematics (Shenzhen), Southern University of Science and Technology, Shenzhen 518055, China
		(\email{zhangz@sustech.edu.cn}).}
}

\usepackage{amsopn}

\usepackage{bm}
\usepackage{float}
\usepackage[caption=false]{subfig}
\usepackage{amssymb}

\newcommand{\dx}{\mathrm{d}x}
\usepackage{amsmath}
\usepackage{enumerate}

\begin{document}
	
\maketitle

\begin{abstract}
    Computing equilibrium shapes of crystals (ESC) is a challenging problem in materials science that involves minimizing an orientation-dependent (i.e., anisotropic) surface energy functional subject to a prescribed mass constraint. The highly nonlinear and singular anisotropic terms in the problem make it very challenging from both the analytical and numerical aspects. Especially, when the strength of anisotropy is very strong (i.e., strongly anisotropic cases), the ESC will form some singular, sharp corners even if the surface energy function is smooth. Traditional numerical approaches, such as the $H^{-1}$ gradient flow, are unable to produce true sharp corners due to the necessary addition of a high-order regularization term that penalizes sharp corners and rounds them off. In this paper, we propose a new numerical method based on the Davis-Yin splitting (DYS) optimization algorithm to predict the ESC instead of using gradient flow approaches. We discretize the infinite-dimensional phase-field energy functional in the absence of regularization terms and transform it into a finite-dimensional constraint minimization problem. The resulting optimization problem is solved using the DYS method which automatically guarantees the mass-conservation and bound-preserving properties. We also prove the global convergence of the proposed algorithm. These desired properties are numerically observed. In particular, the proposed method can produce real sharp corners with satisfactory accuracy. Finally, we present numerous numerical results to demonstrate that the ESC can be well simulated under different types of anisotropic surface energies, which also confirms the effectiveness and efficiency of the proposed method.
\end{abstract}

\begin{keywords}
	phase-field,
    anisotropy,
    optimization approach,
	equilibrium shapes of crystals,
	Davis-Yin splitting
\end{keywords}

\begin{MSCcodes}
	74G15, 74G65, 65Z05
\end{MSCcodes}

\section{Introduction}
Computing equilibrium shapes of crystals (ESC) is an important and centuries-old interface problem arising from materials science. Meanwhile, it is very challenging to design materials with specific functional properties for many nano-technological applications \cite{armelao2006,danielson2006,thompson2012}. The ESC problem can be mathematically described as finding a minimizer of an orientation-dependent (i.e., anisotropic) surface energy functional with a prescribed mass constraint \cite{Gibbs1878}. The geometric construction of ESC was given by the famous Wulff construction in 1901 \cite{wulff1901}, which was rigorously proved by using geometric measure theory~\cite{Fonseca91}. However, the highly nonlinear, singular anisotropic terms in surface energy would bring considerable challenges to theoretical analysis, modeling and numerical simulations. Especially, when the strength of anisotropy is strong enough, some sharp corners will appear in the ESC even if the anisotropic surface energy function is smooth, and this phenomenon is very difficult to capture in numerical simulations~\cite{wise2007solving,torabi2009new,chen2013efficient}.

Generally, there exist two widely-used classes of mathematical models in the literature which can be applied to simulate the ESC, i.e., the sharp-interface models~\cite{Gibbs1878,Jiang17} and the phase-field models \cite{kobayashi1993modeling, torabi2009new}. In this paper, we mainly focus on the phase-field models. To begin with, we consider the interface problem in a fixed, bounded domain $\Omega\subset\mathbb{R}^d$, where $d=2,3$ is the space dimension. Let $\phi$ be a phase variable which takes the values $\pm1$ in the two phases with a smooth transition layer between them. We use the zero level set $\{x \in \Omega: \phi(x,t)=0\}$ to represent the interface curve/surface, i.e., the shape of crystals. We denote the gradient of $\phi$ as $\bm{q}=\nabla \phi$, and the corresponding normal vector as $\bm{n}=(n_1,\cdots,n_d)=\frac{\bm{q}}{|\bm{q}|}$, provided that $|\bm{q}|\neq0$. We consider the following anisotropic Kobayashi-type free energy \cite{kobayashi1993modeling} subject to the mass-conservation constraint:
\begin{equation}\label{energy_functional}
E(\phi)=\int_{\Omega}\left(f(\phi)+\frac{\varepsilon^2}{2}\gamma^2(\bm{n})|\bm{q}|^2\right) \dx, \quad\text{subject to}\quad \int_{\Omega}\phi~\dx=\text{const},
\end{equation}
where $f(\phi)=\frac{\left(\phi^2-1\right)^2}{4}$ is the double-well potential, $\varepsilon$ is a small parameter which controls the thickness of the transition layer, and $\gamma(\bm{n})$ is the surface energy density. For instance, except otherwise specified, we will use the following four-fold 
surface energy density throughout this paper,
\begin{equation}\label{four-fold}
\gamma(\bm{n})=1+\alpha\ (4\sum_{i=1}^{d}n_i^4-3),
\end{equation}
where $\alpha\geqslant0$ controls the strength of anisotropy. More precisely, when $\alpha=0$, the surface energy is isotropic; otherwise, it is anisotropic. In particular, the surface energy becomes strongly anisotropic if $\alpha>\frac{1}{15}$, 
and in this case the ESC will exhibit sharper and sharper corners as $\alpha$ increases. 
It is noteworthy that other types of anisotropic surface energy exist in the literature, for example, the Torabi-type energy~\cite{torabi2009new}, which results in an interface with uniform thickness and independent of orientation. However, due to its complexity, this type of surface energy functional 
is beyond the scope of current work and will be left for our future study.

The commonly used approach to minimizing the anisotropic free energy functional \eqref{energy_functional} subject to the mass-conservation constraint is via the $H^{-1}$ gradient flow induced by the energy functional, i.e., solving the following anisotropic Cahn-Hilliard equation (\cite{wise2007solving,cheng2020weakly}) until its steady state:
\begin{equation}\label{anisotropic_cahn}
\begin{cases}
\partial_t \phi=\Delta \mu,\\
\mu=f'(\phi)-\varepsilon^2 \nabla \cdot\bm{m},\\
\bm{m}=\gamma^2(\boldsymbol{n}) \nabla \phi+\gamma(\boldsymbol{n})|\nabla \phi| (\bm{I}-\bm{n}\otimes\bm{n}) \nabla_n \gamma(\boldsymbol{n}),  
\end{cases}
\end{equation}
where $\nabla_n$ denotes the gradient operator with respect to $\bm{n}$, $\bm{I}$ represents the identity matrix, and $\otimes$ denotes the Kronecker tensor product. It is worth noting that the anisotropic Cahn-Hilliard equation \eqref{anisotropic_cahn} reduces to the classical isotropic Cahn-Hilliard equation when the surface energy function $\gamma(\bm{n})\equiv 1$. Nevertheless, solving the partial differential equation (PDE) \eqref{anisotropic_cahn} would be challenging from both analytical and numerical aspects due to the high nonlinearity and singularity of the anisotropic terms. In particular, one needs to develop energy stable and bound-preserving (on $[-1,1]$) numerical schemes, which is not easy. 
Another critical difficulty arises from the fact that the anisotropic Cahn-Hilliard equation may become ill-posed due to the anti-diffusion in $\nabla\cdot \bm{m}$ when the surface energy density is strongly anisotropic~\cite{wise2007solving}. 

In order to tackle the ill-posedness of the dynamic problem arising from the strongly anisotropic surface energy, a commonly-used approach in the literature is introducing high-order regularization, for instance, adding a bi-harmonic regularization term $\int_\Omega(\Delta\phi)^2\dx$ \cite{wise2007solving} into the Kobayashi-type energy functional \eqref{energy_functional}, or a Willmore regularization term 
$\int_\Omega(\Delta\phi-\frac{1}{\varepsilon^2}f(\phi))^2\dx$ \cite{chen2013efficient,makki2016existence} into the Torabi-type energy functional. However, the high-order regularization terms would bring other difficulties in numerical simulations. More importantly, it would change the ESC for strongly anisotropic systems by penalizing sharp corners and rounding them off in a small length scale, which is unphysical and deviates from experimental observations. Therefore, developing alternative approaches to capturing real sharp corners of the ESC in strongly anisotropic systems is in great demand.

To overcome these difficulties, we consider the direct minimization of the energy functional in 
\eqref{energy_functional} subject to both the mass-conservation and box constraints,
\begin{equation}\label{original_problem}
	\mathop{\arg\min}_\phi\ E(\phi),\quad\ \text{subject to:} \int_{\Omega}\phi \ \dx=\text{constant}~\text{ and }~ \|\phi\|_\infty\leqslant 1.
\end{equation}
This problem can be numerically solved by a discretization-then-optimization approach. 
We first approximate $\phi$ by a finite-dimensional vector and discretize spatial derivatives of all the terms in the energy functional. Then, \eqref{original_problem} is approximated by a finite-dimensional optimization problem, which can be solved efficiently and accurately using finite-dimensional optimization techniques. 
The main idea of this approach is illustrated in \cref{idea}.

\begin{figure}[!htpb]
	\centering
	\includegraphics[width=0.6\linewidth]{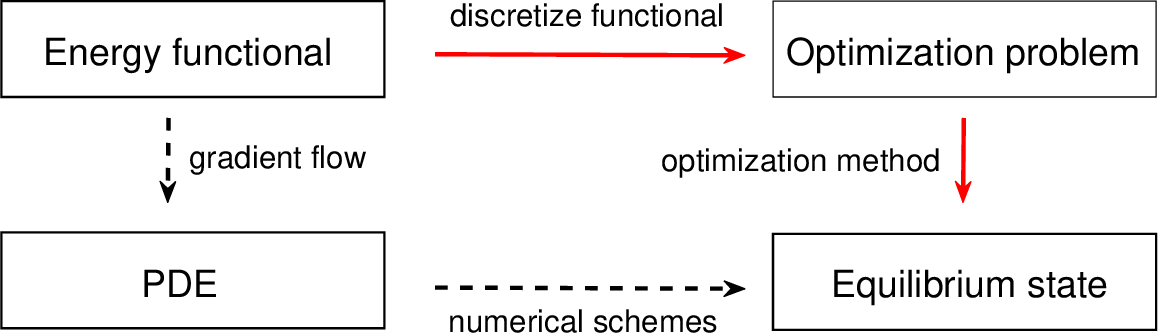}
	\caption{A diagram for the discretization-then-optimization approach versus the gradient flow approach, whose flowcharts are depicted by the red solid lines and the black dash lines respectively.}
	\label{idea}
\end{figure}

The discretization-then-optimization approach is widely used in solving flow control problems \cite{gunzburger1995perspectives}. Hereby it is the first attempt in the literature to apply this approach for obtaining the ESC of anisotropic phase-field models. 
The major difficulty arises from the constrainted non-convex optimization problem after discretization.
To solve this finite-dimensional optimization efficiently with guaranteed convergence, we reformulated the problem into a combination of three properly organized functions and applied the Davis-Yin splitting (DYS) algorithm \cite{davis2017three}.
The DYS algorithm is a three-operator splitting method that was initially designed for convex optimizations, with its convergence being proved as a particular case of fixed-point iteration. Recently this splitting was generalized for non-convex optimizations with guaranteed global convergence to critical points \cite{bian2021three}. For the problem \eqref{original_problem} under our consideration, we decomposed the objective function into the difference of two convex functions (the DC technique in the optimization literature \cite{an2005dc}, or the convex splitting technique in numerical PDE literature \cite{eyre1998unconditionally}), and introduced an indicator function to enforce the constraints. The three split functions were shown to enjoy the desired properties such that the DYS algorithm has global convergence. Moreover, this particularly devised DYS algorithm was efficiently implemented based on fast solvers. As a byproduct, the mass-conservation and the box constraint are satisfied exactly at each step in the iteration.
Various numerical experiments have demonstrated the capability of the proposed numerical method in accurately computing the ESC while maintaining real sharp corners for strongly anisotropic surface energies.

The rest of the paper is organized as follows. In Section 2, we introduce necessary mathematical notations and discretize the ESC problem into a finite-dimensional optimization problem. In Section 3, we present our numerical method based on the DYS algorithm for solving the ESC problem and depict the global convergence result of the algorithm, while leaving detailed proofs in the appendices. Section 4 is devoted to the numerical performance of the proposed method when applied to different anisotropic surface energy densities. Concluding remarks are made in Section 5.

\section{Numerical discretization}
We first introduce necessary notations to discretize the energy functional \eqref{energy_functional}. Consider a rectangular domain $ \Omega=[0,1]\times[0,1] $ in two-dimension (three-dimension is similar). Let $ h=\frac{1}{m} $ be the mesh size, where $ m\in \mathbf{Z}^+ $ is the number of grid points. We discretize $\phi$ as a grid vector $\phi_{i,j}:=\phi(x_i,y_j)$ where $ x_i=(i-\frac{1}{2})h$, $y_j=(j-\frac{1}{2})h $ with $ i,j=1,2,\cdots,m $. This grid vector can be rearranged to be a column vector $\bm{\phi}:=(\phi_1,\phi_2,\cdots,\phi_{m^2})^T$, with $ \phi_{i+m(j-1)}:=\phi_{i,j} $. 

Without loss of generality, we adopt periodic boundary conditions. For simplicity of our presentation, we define the one-dimensional one-sided difference matrix by
\[
D:=\frac{1}{h}\left[ \begin{matrix}
	-1 & 1 & {} & {} & {}  \\
	{} & -1 & 1 & {} & {}  \\
	{} & {} & \ddots  & \ddots  & {}  \\
	{} & {} & {} & -1 & 1  \\
	1 & {} & {} & {} & -1  \\
\end{matrix} \right].
\]
High-order difference matrices can be used if one seeks more accurate approximations. The two-dimensional difference matrices in $ x-$ and $ y-$directions can be defined through Kronecker tensor product:
\[
D_x:={{I}_{m}}\otimes D, \quad D_y:=D\otimes {{I}_{m}},
\]
where $ I_m $ is an $ m\times m $ identity matrix. The discrete negative Laplacian operator is
\begin{equation}\label{descrete_laplace}
    L:=D_x^TD_x+D_y^TD_y.
\end{equation}

With these notations, we can discretize $ \nabla \phi $ as
\begin{equation}\label{def_p}
	\bm{p}=\left[\begin{matrix}
		D_x\\D_y
	\end{matrix}\right]\bm{\phi},
 \quad\mbox{and}\quad
 \bm{p}_k=
	\left[\begin{matrix}
		\bm{e}_k^T &  \\& \bm{e}_k^T 
	\end{matrix}\right]\bm{p}=
	\left[\begin{matrix}
		\bm{e}_k^TD_x\\ \bm{e}_k^TD_y
	\end{matrix}\right]\bm{\phi},\quad k=1,2,\ldots,m^2.
\end{equation}
Here $\bm{p}_k$ is the approximation of $ \nabla \phi $ at each grid point and $\bm{p}$ is their collection,
 $ \bm{e_k}=(0,\cdots,1,\cdots,0)^T $ is an $ m^2\times1 $ canonical vector with only the $ k$-th element being $ 1 $. Then the corresponding unit normal vector at each grid point is
\begin{equation}\label{def_nk}
	\bm{n}_k=\frac{\bm{p}_k}{|\bm{p}_k|},\quad k=1,2,\ldots,m^2\text{, when }\bm{p}_k\ne\bm{0}.
\end{equation}
We set $ \bm{n}_k=\bm{0} $ when $ \bm{p}_k=\bm{0} $, which leads to a continuous gradient of the discrete energy function. 

Consequently, the finite-dimensional discretization of \eqref{original_problem} becomes
\begin{equation}\label{discrete problem}
    \begin{split}
        \mathop{\arg\min}_{\bm{\phi}}\ E_h(\bm{\phi})=\sum_{k=1}^{m^2}\left(f(\phi_k)+\frac{\varepsilon^2}{2}\gamma^2(\bm{n}_k)|\bm{p}_k|^2\right),\\ \text{subject to: } \bm{1}^T\bm{\phi}=\text{constant}~\text{ and }~ \|\bm{\phi}\|_\infty\leqslant 1.
    \end{split}
\end{equation}
where $ \bm{1}=(1,1,\cdots,1)^T $ is an $ m^2\times1 $ vector.

\section{Numerical method}\label{section3}
In this section, we explore the Davis-Yin splitting (DYS) algorithm for simulating the ESC problem. The DYS algorithm, originally introduced as a three-operator splitting technique in \cite{davis2017three} for convex optimizations, was recently generalized for non-convex problems with its global convergence established in~\cite{bian2021three}. In Section~\ref{section2}, we first introduce the generic DYS algorithm and give its convergence results.
Then, we propose the numerical method for solving the ESC problem based on the DYS algorithm in Section~\ref{section5}. Its convergence is verified through the validation of the key properties of the split functions in Section~\ref{section6}. 

\subsection{The Davis-Yin splitting algorithm}\label{section2}
The DYS algorithm aims to solve an optimization problem in the form of
\begin{equation}\label{opt3}
	\mathop{\arg\min}_{\bm{x}\in\mathbb{R}^n}~ F(\bm{x})+G(\bm{x})+H(\bm{x}),
\end{equation}
where $ F,~G,$ and $~H$ satisfy Assumption~\ref{assumption} shown below.
\begin{assumption}\label{assumption}
The functions $F$ and $H$ are Lipschitz continuously differentiable, i.e., there exist positive constants $L_{_{F}},~L_{_{H}}$ such that
		\[
		\begin{aligned}
			\left\|\nabla F\left(\bm{x}_1\right)-\nabla F\left(\bm{x}_2\right)\right\| \leq L_{_{F}}\left\|\bm{x}_1-\bm{x}_2\right\| \quad \forall \bm{x}_1, \bm{x}_2 \in \mathbb{R}^n,\\
			\left\|\nabla H\left(\bm{x}_1\right)-\nabla H\left(\bm{x}_2\right)\right\| \leq L_{_{H}}\left\|\bm{x}_1-\bm{x}_2\right\| \quad \forall \bm{x}_1, \bm{x}_2 \in \mathbb{R}^n,
		\end{aligned}
		\]
and the function $G$ is an indicator function of a nonempty closed convex set. 
		
\end{assumption}
The DYS algorithm is stated in~\cref{alg:algorithm1}.
\begin{algorithm}
	\caption{DYS for problem \eqref{opt3}}
	\label{alg:algorithm1}
	
	
 \begin{algorithmic}[1]
 \REQUIRE An initial $ \bm{x}^{(0)} $ and a step size $ \tau $;
 \WHILE {a termination criterion is not met,}
 \STATE Set {\small
	\begin{align}
		&\bm{y}^{(n+1)} \in \mathop{\arg\min}_{\bm{y}}\left\{F(\bm{y})+\frac{1}{2 \tau}\left\|\bm{y}-\bm{x}^{(n)}\right\|^2\right\}, \label{step1_alg1} \\
		&\bm{z}^{(n+1)} \in \arg \min _{\bm{z}}\left\{G(\bm{z})+\frac{1}{2 \tau}\left\|\bm{z}-\left(2 \bm{y}^{(n+1)}-\tau \nabla H(\bm{y}^{(n+1)})-\bm{x}^{(n)}\right)\right\|^2\right\}, \label{step2_alg1}\\
		&\bm{x}^{(n+1)}=\bm{x}^{(n)}+(\bm{z}^{(n+1)}-\bm{y}^{(n+1)}) \label{step3_alg1};
	\end{align} }
\vspace{-1em}
 \STATE $n = n + 1$;
 \ENDWHILE
 \end{algorithmic}
\end{algorithm}

The first-order optimality conditions for the subproblems in \cref{alg:algorithm1} are
\begin{align}
	&0 = \nabla F(\bm{y}^{(n+1)})+\frac{1}{\tau}(\bm{y}^{(n+1)}-\bm{x}^{(n)}),\label{first_optimal} \\
	&0 \in \partial G(\bm{z}^{(n+1)})+\frac{1}{\tau}(\bm{z}^{(n+1)}-2 \bm{y}^{(n+1)}+\tau \nabla H(\bm{y}^{(n+1)})+\bm{x}^{(n)}),	\label{second_optimal}
\end{align}
where $\partial$ denotes the Clarke generalized subdifferential \cite{clarke1990optimization}.
Combining \eqref{first_optimal} and \eqref{second_optimal} together, we have
\begin{equation}\label{optimal_condtion}
	\frac{\bm{y}^{(n+1)}-\bm{z}^{(n+1)}}{\tau}\in \nabla F(\bm{y}^{(n+1)})+\partial G(\bm{z}^{(n+1)})+\nabla H(\bm{y}^{(n+1)}).
\end{equation}
This implies that if the sequence $\{x^{(n+1)},~y^{(n+1)},~z^{(n+1)}\}_{n\geqslant 0}$ has a cluster point $(x^*,~y^*,~z^*)$ and $ \lim\limits_{n\to\infty}\|\bm{y}^{(n+1)}-\bm{z}^{(n+1)}\|=0 $, then $ \bm{z}^* $ is a critical point of Problem~\eqref{opt3}. 

It has been shown in~\cite{bian2021three} that the DYS algorithm is a descent algorithm with respect to the function
\begin{equation}\label{new_energy}
	\begin{split}
		\Theta_\tau(\bm{x}, \bm{y}, \bm{z})=& F(\bm{y})+G(\bm{z})+H(\bm{y})+\frac{1}{2 \tau}\|2 \bm{y}-\bm{z}-\bm{x}-\tau \nabla H(\bm{y})\|^2 \\
		&-\frac{1}{2 \tau}\|\bm{x}-\bm{y}+\tau \nabla H(\bm{y})\|^2-\frac{1}{\tau}\|\bm{y}-\bm{z}\|^2
	\end{split}
\end{equation}
for certain $\tau > 0$ in the sense that 
\begin{equation}\label{descent_lemma}
    	\quad \Theta_\tau\left(\bm{x}^{(n+1)}, \bm{y}^{(n+1)}, \bm{z}^{(n+1)}\right)-\Theta_\tau\left(\bm{x}^{(n)}, \bm{y}^{(n)}, \bm{z}^{(n)}\right) \leqslant-\mathcal{D}(\tau)\left\|\bm{y}^{(n+1)}-\bm{y}^{(n)}\right\|^2.
\end{equation}
There exists a threshold $\bar{\tau}$ such that $\mathcal{D}(\tau) >0$ if $0<\tau<\bar{\tau}$. Especially, if $ F $ is convex, then the threshold $\bar{\tau}$ is
	\begin{equation}\label{tau_threshold}
		\bar{\tau}=\min \left\{\frac{1}{L_{_{F}}}, \frac{-(2 L_{_{H}}+6 L_{_{F}})+\sqrt{(2 L_{_{H}}+6 L_{_{F}})^2+12 L_{_{H}} L_{_{F}}}}{6 L_{_{H}} L_{_{F}}}\right\}.
	\end{equation}

The global convergence of the DYS algorithm is given by the following theorem.
\begin{theorem}[global convergence of the whole sequence \cite{bian2021three}]\label{lemma_globalconvergence}
	Let  \cref{assumption} be satisfied and let the parameter $\tau$ in \cref{alg:algorithm1} be such that $\mathcal{D}(\tau)>0$. Let $\left\{\left(\bm{x}^{(n)}, \bm{y}^{(n)}, \bm{z}^{(n)}\right)\right\}_{n \geqslant 1}$ be a sequence generated by \cref{alg:algorithm1} which has a cluster point  $\left(\bm{x}^*, \bm{y}^*, \bm{z}^*\right)$. If the functions $F$, $G$ and $H$ are semi-algebraic (see  \cref{appe_semi}), then the following statements hold:
    \begin{enumerate}[(i)]
        \item  $0 \in \nabla F\left(\bm{z}^*\right)+\partial G\left(\bm{z}^*\right)+\nabla H\left(\bm{z}^*\right)$, i.e., $z^*$ is a critical point.

        \item The limit $\lim _{n \rightarrow \infty} \Theta_\tau\left(\bm{x}^{(n)}, \bm{y}^{(n)}, \bm{z}^{(n)}\right)$ exists and
        \[
	\Theta^*:=\lim _{t \rightarrow \infty} \Theta_\tau\left(\bm{x}^{(n)}, \bm{y}^{(n)}, \bm{z}^{(n)}\right)=\Theta_\tau\left(\bm{x}^*, \bm{y}^*, \bm{z}^*\right).
	\]

        \item $\sum_{n \geqslant 1}\left\|\bm{x}^{(n+1)}-\bm{x}^{(n)}\right\|,~ \sum_{n \geqslant 1}\left\|\bm{y}^{(n+1)}-\bm{y}^{(n)}\right\|~\text{and}~ \sum_{n \geqslant 1}\left\|\bm{z}^{(n+1)}-\bm{z}^{(n)}\right\|$ are all convergent series.
    \end{enumerate}
\end{theorem}

\subsection{DYS algorithm for the ESC problem} \label{section5}
Motivated by the convex splitting technique \cite{eyre1998unconditionally} and the linear stabilization method \cite{shen2010numerical} in numerical methods of phase-field equations, or the DC technique \cite{an2005dc} in the optimization literature, we propose the following splitting of the objective function:
\begin{equation}\label{splitting form}
	\begin{aligned}
		F(\bm{\phi})&=a\frac{\varepsilon^2 }{2}{{\bm{\phi} }^{T}}L\bm{\phi}+\frac{b}{2}\sum_{k=1}^{m^2}\phi_k^2,\\
		G(\bm{\phi})&=\delta_C(\bm{\phi}),\\
		H(\bm{\phi})&=\sum_{k=1}^{m^2}\left(f(\phi_k)-\frac{b}{2}\phi_k^2+\frac{\varepsilon^2}{2}\gamma^2(\bm{n}_k)|\bm{p}_k|^2\right)-a\frac{\varepsilon^2 }{2}{{\bm{\phi} }^{T}}L\bm{\phi},
	\end{aligned}
\end{equation}
where $ a $ and $ b $ are positive constants such that $ H(\bm{\phi}) $ is a concave function in the set $C$ defined by
\begin{equation}\label{setC}
	C= \{\bm{\phi}:\bm{1}^T\bm{\phi}=V_0=\text{constant}, \text{ and }\|\bm{\phi}\|_\infty\leqslant1\},
\end{equation}
and $\delta_C(\bm{\phi})$ is its indicator function, i.e., $\delta_C(\bm{\phi})=0$ if $ \bm{\phi}\in C $, otherwise $ \delta_C(\bm{\phi})=+\infty $.


Since $ F $ and $ G $ are convex functions, the solutions of Subproblems~\eqref{step1_alg1} and~\eqref{step2_alg1} are unique. The first-order optimality condition of Subproblem~\eqref{step1_alg1} yields 
\[
\frac{\bm{y}^{(n+1)}-\bm{x}^{(n)}}{\tau}=-(a\varepsilon^2L+bI_{m^2})\bm{y}^{(n+1)},
\]
which can be solved efficiently by the fast Fourier transformation (FFT).
The solution of Subproblem \eqref{step2_alg1} is computed in a closed form:
\[
\bm{z}^{(n+1)}=\operatorname{P}_C\left(2 \bm{y}^{(n+1)}-\tau \nabla H\left(\bm{y}^{(n+1)}\right)-\bm{x}^{(n)}\right),
\]
where $ \operatorname{P}_C $ is the projection operator of $C$ that is given by the following lemma~\cite{beck2017first}.
\begin{lemma}[Projection onto the set $ C $ \cite{beck2017first}]\label{projection lemma}
	If $ C $ is a set defined in \eqref{setC}, then
	\[{\operatorname{P}_{C}}(\bm{\phi})={\operatorname{P}_{Box[\mathbf{-1,1}]}}(\bm{\phi}-{{\lambda }^{*}}\mathbf{1})=(\min \left\{ \max \left\{ {{(\bm{\phi}-{{\lambda }^{*}}\mathbf{1})}_{i}},-1 \right\},1 \right\})_{i=1}^{m^2},\]
	where ${{\lambda }^{*}}$ is a root of the equation
	\begin{equation}\label{equ:projection}
		\tilde{f}(\lambda) := {{\mathbf{1}}^{T}}{\operatorname{P}_{Box[\mathbf{-1,1}]}}(\bm{\phi}-\lambda \mathbf{1})-V_0 = 0,
	\end{equation}
	and the projection $ \operatorname{P}_{Box[\mathbf{-1,1}]}(\bm{\psi}) $ is a cut-off operator, i.e.,  for any $ \bm{\psi}\in\mathbb{R}^{m^2} $, $$ \operatorname{P}_{Box[\mathbf{-1,1}]}(\bm{\psi})=(\min \left\{ \max \left\{ \psi_i,-1 \right\},1 \right\})_{i=1}^{m^2}. $$ 
\end{lemma}

Since the scalar function $\tilde{f}(\lambda)$ is nonincreasing, its root $\lambda^*$ can be found efficiently. In this paper, the bisection method is used for the numerical computation of $\lambda^*$.


Computing the gradient of $ H $  is a little cumbersome. We provide it in Lemma~\ref{lemma_gradient_H}  and leave its proof in \cref{appe_gradient_h}.

\begin{lemma}[The gradient of $ H(\bm{\phi}) $] \label{lemma_gradient_H}
	The gradient of $ H(\bm{\phi}) $ is given by
	\begin{equation}\label{gradient_h}
		\nabla  H (\bm{\phi})=\left[
		\begin{matrix}
			f'(\phi_1)-b\phi_1 \\f'(\phi_2)-b\phi_2 \\ \cdots\\f'(\phi_{m^2}) -b\phi_{m^2}
		\end{matrix}
		\right]+\varepsilon^2\left(
		D_x^T\left[\begin{matrix}
			\mathcal{A}_1(\bm{p}_1)\\
			\mathcal{A}_1(\bm{p}_2)\\
			\cdots\\
			\mathcal{A}_1(\bm{p}_{m^2})
		\end{matrix}\right]+
		D_y^T\left[\begin{matrix}
			\mathcal{A}_2(\bm{p}_1)\\
			\mathcal{A}_2(\bm{p}_2)\\
			\cdots\\
			\mathcal{A}_2(\bm{p}_{m^2})
		\end{matrix}\right]
		\right),
	\end{equation}
	where  $ \mathcal{A}_1(\bm{p}_k) $ and $ \mathcal{A}_2(\bm{p}_k) $ are scalar components defined through
	\[
	\left[\begin{matrix}
		\mathcal{A}_1(\bm{p}_k)\\
		\mathcal{A}_2(\bm{p}_k)
	\end{matrix}\right]:=	(\gamma^2(\bm{n}_k)-a)\bm{p}_k+\gamma(\bm{n}_k)|\bm{p}_k|(I_2-\bm{n}_k\otimes\bm{n}_k)\nabla_{\bm{n}}\gamma(\bm{n}_k),\quad   k=1,2,\cdots,m^2.
	\]
\end{lemma}
Using these two lemmas, we can compute the ESC through \cref{alg:algorithm1} with a prescribed threshold $\bar\tau$ for the step size. 

However, the theoretical value of $\bar{\tau}$ given in \eqref{tau_threshold} is too restricted due to its complexity  and may not be useful 
in practice.
We adopt a more practical strategy for variably selecting $\tau$ as follows \cite{sun2015convergent,li2016douglas,bian2021three}:

{\it We initialize the algorithm with a large $\tau_0$. If $ \tau>\bar{\tau} $, and the iteration satisfies either $ \|\bm{y}^{(n+1)}-\bm{y}^{(n)}\|>c_0/n $ or $ \|\bm{y}^{(n+1)}\|> c_1  $ for some predefined $ c_0,~c_1 > 0 $, then we reduce $\tau$ by a factor $\frac12$. }

In the worst case, $ \tau\leqslant\bar{\tau} $ after finitely many decreases which ensures the convergence of the sequence by \cref{lemma_globalconvergence}. Otherwise, we have $ \|\bm{y}^{(n+1)}-\bm{y}^{(n)}\|\leqslant c_0/n $ and $ \|\bm{y}^{(n+1)}\|\leqslant c_1  $ for all sufficiently large $ n $. 
 Combining the fact that $ \bm{z}^{(n+1)} $ is bounded and \eqref{first_optimal}, we know that the sequence $\left\{\left(\bm{x}^{(n)}, \bm{y}^{(n)}, \bm{z}^{(n)}\right)\right\}_{n \geqslant 1}$ is bounded which leads to the existence of cluster points. In addition, \eqref{step3_alg1} and \eqref{first_optimal} imply that $\|\bm{y}^{(n+1)}-\bm{z}^{(n+1)}\|
 \leqslant(1+\tau L_{_{F}})\|\bm{y}^{(n+1)}-\bm{y}^{(n)}\| $, leading to the convergence of the sequence by \eqref{optimal_condtion}.
 In fact, this practical step-size strategy can not only accelerate the convergence speed numerically but also have the potential to avoid the iteration getting stuck at `bad' critical points \cite{li2016douglas}.


The implemented DYS algorithm for computing the ESC is given in \cref{alg:algorithm3}.
%
%
%
%
%
%

\begin{algorithm}
	\caption{DYS for \eqref{splitting form} with dynamic step size}
	\label{alg:algorithm3}
	\begin{algorithmic}[1]
		\STATE{Initialize $ \bm{y}^{(0)} $, $ \tau=\tau_0 $, $ \bm{x}^{(0)}=\bm{y}^{(0)}+\tau(a\varepsilon^2L+bI_{m^2})\bm{y}^{(0)} $ and parameters $(c_0, c_1) $.}
		
		\FOR{$ n=0,1,\cdots $,}
		\STATE{\label{step1_alg3} Use FFT to solve $ \frac{\bm{y}^{(n+1)}-\bm{x}^{(n)}}{\tau}=-(a\varepsilon^2L+bI_{m^2})\bm{y}^{(n+1)} $ for $ \bm{y}^{(n+1)} $.}
		
		\STATE{\label{step2_alg3}  $ \bm{z}^{(n+1)}=\operatorname{P}_C\left(2 \bm{y}^{(n+1)}-\tau \nabla H\left(\bm{y}^{(n+1)}\right)-\bm{x}^{(n)}\right) $.}
		
		\STATE{\label{step3_alg3}  $ \bm{x}^{(n+1)}=\bm{x}^{(n)}+\left(\bm{z}^{(n+1)}-\bm{y}^{(n+1)}\right) $.}
		
		\STATE{If a termination criterion is met, break.}
		
		\STATE{Update $\tau$: if  $ \|\bm{y}^{(n+1)}-\bm{y}^{(n)}\|> c_0/n $ or $ \|\bm{y}^{(n+1)}\|> c_1  $, then $ \tau\leftarrow\tau/2 $.}
		\ENDFOR
		\RETURN  $ \bm{z}^{(n+1)} $.
	\end{algorithmic}
\end{algorithm}

\subsection{Convergence property}  \label{section6}
The original double well potential $ f(\phi)=\frac{(\phi^2-1)^2}{4} $ does not have Lipschitz continuous gradient. This presents difficulties in the convergence analysis of the algorithm due to the violation of \cref{assumption}. A commonly used modification is to truncate the double-well potential $ f(\phi) $ and concatenate it with a quadratic growth function \cite{shen2010numerical}, i.e.,
\begin{equation} \label{equ:truef}
\hat{f}(\phi)= \begin{cases}\frac{3 M^2-1}{2} \phi^2-2 M^3 \phi+\frac{1}{4}\left(3 M^4+1\right), & \phi>M, \\ \frac{1}{4}\left(\phi^2-1\right)^2, & \phi \in[-M, M], \\ \frac{3 M^2-1}{2} \phi^2+2 M^3 \phi+\frac{1}{4}\left(3 M^4+1\right), & \phi<-M,\end{cases}
\end{equation}
for some positive number $M\geqslant1$. Then there exists a constant $ L_f $ such that 
\begin{equation}\label{smoothnness_doublewell}
	\max_{\phi\in\mathbb{R}}|\hat{f}''(\phi)|\leqslant L_f.
\end{equation}


Sufficient numerical experiments (\cite{wise2007solving,shen2010numerical,chen2019fast,cheng2020weakly}) show that the energy functional with this modified potential yields bounded equilibrium solutions with bounds nearby $\pm1$. Moreover, the box constraint in \cref{discrete problem} also guarantees the boundedness of the equilibrium profiles.
This suggests that we can take $ M=1 $ throughout this paper.
Nevertheless, numerical experiments indicate that the proposed algorithm also works well even if we take the original double-well potential $ f(\phi)=\frac{1}{4}(\phi^2-1)^2 $. 


As a major technical result, we will show that the sequence generated by the DYS algorithm associated with the splitting \cref{splitting form} converges to a critical point of \cref{discrete problem}. To this end, we verify the assumptions of \cref{lemma_globalconvergence} for the ESC problem. 
The first condition is the existence of a clustering point, which can be ensured by the boundedness of the sequence generated by the DYS algorithm. A sufficient condition for the boundedness has been given in~\cite{bian2021three}, which is unfortunately not satisfied for the ESC problem. Hereby we show the boundedness in Theorem~\ref{lemma_boundness} by further exploring the particular structure of the ESC problem and mimicing the analysis in~\cite{bian2021three}. Its proof is provided in \cref{appe_lemma_bound}.


\begin{theorem}[Boundedness]\label{lemma_boundness}
	Let \cref{assumption} hold and the parameter $\tau$ in \cref{alg:algorithm1}  satisfy $\mathcal{D}(\tau)>0$. Suppose that the function $F$ is bounded below and coercive (i.e., $ \|F(\bm{x})\|\rightarrow+\infty $, if $ \|\bm{x}\|\rightarrow\infty $), the function $ G $ is an indicator function of a bounded set, and the function $ H $ is a concave function. Then, the sequence $\left\{\left(\bm{x}^{(n)}, \bm{y}^{(n)}, \bm{z}^{(n)}\right)\right\}_{n \geqslant 1}$ generated by  \cref{alg:algorithm1} is bounded.
\end{theorem}

Based on \cref{lemma_globalconvergence} and \cref{lemma_boundness}, we can summarize the conditions that need to be verifed as follows:
\begin{enumerate}[(i)] 
	\item\label{condition1} $ F $ is bounded below, coercive and has a Lipschitz continuous gradient.
	
	\item\label{condition2} $ G $ is an indicator function of a nonempty closed convex set.
	
	\item\label{condition3} $ H $ is a concave function and has a Lipschitz continuous gradient;
	
	\item\label{condition4} $ F $, $ G $ and $ H $ are semi-algebraic.
\end{enumerate}	

For \eqref{condition1}, since $ L $ is positive definite, it follows that $ F $ is bounded below. The coercivity of $ F $ is obvious due to its quadratic form. Moreover, because $ \|\nabla^2 F(\bm{\phi})\|_2\leqslant a\varepsilon^2\|L\|_2+b $, $ F $ has a Lipschitz continuous gradient.

\eqref{condition2} is a direct consequence of the fact that the bounded set $ C $ is the intersection of a box and a hyperplane. 

To verify \eqref{condition3}, we first split $H$ into $H(\bm{\phi})=h_1(\bm{\phi})+h_2(\bm{\phi})$, with $h_1$ being a function solely depending on $\bm{\phi}$ and $h_2$ depending only on its first-order finite differences: 
    \begin{equation}\label{eq_h_1}
        	h_1(\bm{\phi})=\sum_{k=1}^{m^2}\left(\hat{f}(\phi_k)-\frac{b}{2}\phi_k^2\right),\quad h_2(\bm{\phi})=\sum_{k=1}^{m^2}\frac{\varepsilon^2}{2}(\gamma^2(\bm{n}_k)-a)|\bm{p}_k|^2,
    \end{equation}
	where we have used the identity ${{\bm{\phi} }^{T}}L\bm{\phi}=\sum_{k=1}^{m^2}|\bm{p}_k|^2$.

An insightful observation is that the summands in $h_2$ can always be written in terms of some fractional functions if $\gamma^2$ is a polynomial. For example, if $\gamma(\bm{n})=n_1$, then  $\gamma^2(\bm{n})|\bm{q}|^2=\frac{q_1^2|\bm{q}|^2}{|\bm{q}|^2}$. Thus to analyze $h_2$, it suffices to study the function $\zeta_{_l}$ of the following form:
\begin{equation} \label{equ:zeta}
\zeta_{_l}(\bm{x})= \begin{cases}\frac{\eta_{_{l+2}}(\bm{x})}{|\bm{x}|^{l}}, & \bm{x}\in\mathbb{R}^2\backslash\{\bm{0}\}, \\ 0, & \bm{x}=\bm{0},\end{cases}
\end{equation}
where $ \eta_{_{l+2}}(\bm{x}) $ is a homogeneous polynomial with degree $ l+2 $. This function has the following properties whose proof is given in \cref{appe_rearrange}.

\begin{lemma}\label{lemma_rearrange}
	Let $\zeta_{_l}(\bm{x})$ be defined by \eqref{equ:zeta}.  Then $\nabla_{\bm{x}} \zeta_{_l}$ is continuous, and  there exists a constant $ C_l $ such that $ \left| \frac{\partial^2\zeta_{_l}}{\partial x_i\partial x_j} \right|\leqslant C_l $ for any  $ \bm{x}\in\mathbb{R}^2\backslash\{\bm{0}\} $ and $ i,j=1,2$.
\end{lemma}

Using \eqref{equ:zeta}, if $\gamma^2$ is a $ (q-1) $-degree polynomial, we can represent $ h_2(\bm{\phi}) $ as
	\begin{equation}\label{h2represent_p}
h_2(\bm{\phi})=\sum_{k=1}^{m^2}\frac{\varepsilon^2}{2}\sum_{l=0}^{q-1}\zeta_{_l}(\bm{p}_k).
	\end{equation}
Then we can verify the condition \eqref{condition3} in the following lemma.

\begin{lemma}[H is concave and has a Lipschitz continuous gradient]\label{h_smoothness}
	Let $\gamma^2(\bm{n})$ be a polynomial. Then $ H $ has a Lipschitz continuous gradient. Moreover, if the splitting parameter $ a\geqslant\max_l \|\nabla^2_{\bm{x}}\zeta_{_{l}}\|_2 $ with $\zeta_{_{l}}$'s defined in \eqref{h2represent_p}, and $ b\geqslant L_f
 $,  then $ H(\bm{\phi}) $ is a concave function.
\end{lemma}
\begin{proof} 
It suffices to show that both $ h_1$ and $h_2$  are concave and have Lipschitz continuous gradients. 
Since  $ \nabla^2h_1(\bm{\phi}) $ is a diagonal matrix whose $ k $-th entry is $ \hat{f}''(\phi_k)-b $. By \eqref{smoothnness_doublewell}, we have $ h_1(\bm{\phi}) $ is concave and has a Lipschitz continuous gradient.


The differential of $ h_2(\bm{\phi}) $ can be calculated through chain rule as
        \begin{equation}\label{grad_phi}
        \nabla h_2(\bm{\phi})
	=\frac{\varepsilon^2}{2}\sum_{k=1}^{m^2}\sum_{l=0}^{q-1}
	\left[\begin{matrix}
		D_x^T & D_y^T
	\end{matrix}\right]
	\left[\begin{matrix}
		\bm{e_k} & \\& \bm{e_k} 
	\end{matrix}\right]
	\nabla_{\bm{x}}\zeta_{_l}(\bm{p}_k).  
        \end{equation}
	For any $ \bm{\phi}, \bm{\bar{\phi}}\in \mathbb{R}^{m^2} $, we can apply the notations in \eqref{def_p} to define their corresponding discrete gradients $\bm{p},\bm{\bar{p}}$ and $\bm{p}_k,\bm{\bar{p}}_k$.
	By \cref{descrete_laplace}, the following estimate can be established:
	\begin{equation}\label{pk_norm}
		\|\bm{p}_k-\bm{\bar{p}}_k\|_2
		\leqslant\left\|	\left[\begin{matrix}
			\bm{e_k}^T &  \\& \bm{e_k}^T 
		\end{matrix}\right]\right\|_2~
		\left\|	\left[\begin{matrix}
			D_x\\D_y
		\end{matrix}\right] \right\|_2
		\|\bm{\phi}-\bm{\bar{\phi}}\|_2
		=\|L\|_2^{\frac{1}{2}}\|\bm{\phi}-\bm{\bar{\phi}}\|_2.
	\end{equation}
 
     To show $ h_2 $ has a Lipschitz continuous gradient, we have an estimate that
	\begin{equation}\label{h2_identity}
		\begin{aligned}
			\|\nabla h_2(\bm{\phi})-\nabla h_2(\bm{\bar{\phi}})\|_2
		  \leqslant&\frac{q\varepsilon^2}{2}\|L\|^{\frac{1}{2}}_2\cdot\max_l
            \bigg\|\sum_{k=1}^{m^2}\left[\begin{matrix}
		  \bm{e_k} & \\& \bm{e_k} 
	       \end{matrix}\right]  \bigg(
	        \nabla_{\bm{x}}\zeta_{_l}(\bm{p}_k)-\nabla_{\bm{x}}\zeta_{_l}(\bar{\bm{p}}_k)\bigg)\bigg\|_2\\
			=&\frac{q\varepsilon^2}{2}\|L\|_2^{\frac{1}{2}}\cdot\max_{k,i,l}\bigg|\frac{\partial \zeta_{_l}}{\partial x_{i}}(\bm{p}_k)-\frac{\partial \zeta_{_l}}{\partial x_{i}}(\bm{\bar{p}}_k)\bigg|.
		\end{aligned}
	\end{equation}
	
	On the one hand, if the line segment from $ \bm{p}_k $ to $ \bm{\bar{p}_k} $ does not pass through the origin $ \bm{0} $, then by mean value theorem, \cref{lemma_rearrange} and \eqref{pk_norm}, we have that
	\begin{equation}\label{idetity_h2_1}
		\bigg|\frac{\partial \zeta_{_l}}{\partial x_{i}}(\bm{p}_k)-\frac{\partial \zeta_{_l}}{\partial x_{i}}(\bm{\bar{p}}_k)\bigg|
		\leqslant C_l~\|\bm{p}_k-\bm{\bar{p}}_k\|_2\leqslant C_l~\|L\|_2^{\frac{1}{2}}~\|\bm{\phi}-\bm{\bar{\phi}}\|_2.
	\end{equation}
	
	On the other hand, if the line segment from $ \bm{p}_k $ to $ \bm{\bar{p}_k} $ pass through the origin $ \bm{0} $, we have $ \|\bm{p}_k\|_2+\|\bm{\bar{p}}_k\|_2=\|\bm{p}_k-\bm{\bar{p}}_k\|_2 $. Again by mean value theorem, it follows
	\begin{equation}\label{idetity_h2_2}
		\begin{aligned}
			\bigg|\frac{\partial \zeta_{_l}}{\partial x_{i}}(\bm{p}_k)-\frac{\partial \zeta_{_l}}{\partial x_{i}}(\bm{\bar{p}}_k)\bigg|
			\leqslant&
			\bigg|\frac{\partial \zeta_{_l}}{\partial x_{i}}(\bm{p}_k)-\frac{\partial \zeta_{_l}}{\partial x_{i}}(\bm{0})\bigg|
			+\bigg|\frac{\partial \zeta_{_l}}{\partial x_{i}}(\bm{0})-\frac{\partial \zeta_{_l}}{\partial x_{i}}(\bm{\bar{p}}_k)\bigg|	\\
			\leqslant& C_l(\|\bm{p}_k\|_2+\|\bm{\bar{p}}_k\|_2)
            =C_l~\|L\|_2^{\frac{1}{2}}~\|\bm{\phi}-\bm{\bar{\phi}}\|_2.
		\end{aligned}
	\end{equation}

 Combining \eqref{h2_identity}, \eqref{idetity_h2_1} and \eqref{idetity_h2_2}, we obtain
	\[
	\|\nabla h_2(\bm{\phi})-\nabla h_2(\bm{\bar{\phi}})\|_2\leqslant \frac{Cq\varepsilon^2}{2} \|L\|_2~\|\bm{\phi}-\bm{\bar{\phi}}\|_2,
	\]
	where $ C=\max_{l}C_l $. Hence, $ h_2 $ has a Lipschitz continuous gradient. 
	
    
    At last, we prove $ -h_2 $ is convex. Since the composition of a convex function and an affine function is still convex, by \eqref{eq_h_1} we only need to show $(a-\gamma^2(\bm{n}_k))|\bm{p}_k|^2$ is convex with respect to $\bm{p}_k$. Because $\gamma^2$ is a polynomial function and the summations of convex functions are still convex, it suffices to prove $\hat{h}_2(\bm{x}):=a|\bm{x}|^2-\zeta_{_l}(\bm{x})$ is convex for any $\bm{x}\in\mathbb{R}^2$. This is easily verified since its Hessian matrix $\nabla^2_{\bm{x}}\hat{h}_2=aI-\nabla^2_{\bm{x}}\zeta_{_{l}}$ is 
    positive-definite in $\mathbb{R}^2\backslash\{\bm{0}\}$.
    Combining the continuity of $\nabla_{\bm{x}}\hat{h}_2$ and Lemma 2.4 of \cite{cheng2020weakly}, we obtain that $\hat{h}_2$ is convex which implies $h_2$ is concave.
\end{proof}

\begin{remark}
	In this lemma, we have assumed  $\gamma^2$ is a polynomial. This form includes a broad class of anisotropy functions:
	\begin{enumerate}
		\item A function of the form: $\gamma(\bm{n})=\alpha_0+\sum_{i=1}^{d}\alpha_{_{1,i}}n_i+\sum_{i=1}^{d}\alpha_{_{2,i}}n_i^2+\cdots$. This includes two-, three-, four-, six-fold anisotropy functions \cite{wise2007solving,cheng2020weakly}. 
		
		\item  Riemannian metric form \cite{deckelnick2005computation}: 
		$
		\gamma(\mathbf{n})= \sqrt{R \mathbf{n} \cdot \mathbf{n}}
		$
		where the matrix $ R $ is a symmetric positive definite matrix. This expression includes ellipsoidal anisotropy \cite{zhao2020parametric} and some Riemannian metric anisotropy as its particular cases \cite{jiang2019sharp}.
	\end{enumerate}
\end{remark}

\begin{remark}
    The condition $ a\geqslant\max_l \|\nabla^2_{\bm{x}}\zeta_{_{l}}\|_2 $  can be relaxed for $\gamma(\bm{n})$ with given explicit expressions. For example, an estimate for four-fold anisotropy \eqref{four-fold} has been established in Ref.~\cite{cheng2020weakly}. In practice, $a$ can also be adjusted by a close observation in the optimality conditions or numerical convergence trends.
\end{remark}


Finally, we verify the condition \eqref{condition4}.
\begin{lemma}[Semi-algebraic function]
	Let $\gamma^2$ be a polynomial. Then $ F $, $ G $ and $ H $ are semi-algebraic.
\end{lemma}

\begin{proof}
	Items 1 and 2 in \cref{example} implies $ F $, $ G $ and $ h_1 $ are semi-algebraic. The remaining problem is to prove that $ h_2$ is also semi-algebraic. 
    By items 3 and 4 in \cref{example} and 
 \eqref{h2represent_p}, we only need to show $\zeta_{_l}$
	is semi-algebraic. This is easily seen from the definition since the graph of $ \zeta_{_l}(\bm{p}_k) $ can be written as
	\[
	\bigg\{
	(\bm{p}_k,t)\in\mathbb{R}^d\times\mathbb{R}:t=\frac{\eta_{_{l+2}}(\bm{p}_k)}{|\bm{p}_k|^{l}}
	\bigg\}
	=
	\bigg\{
	(\bm{p}_k,t)\in\mathbb{R}^d\times\mathbb{R}:t^2|\bm{p}_k|^{2l}-\eta_{_{l+2}}^2(\bm{p}_k)=0
	\bigg\}.
	\]
\end{proof}

Now we can summarize our main results in the following theorem.
\begin{theorem}[global convergence for the ESC problem]\label{main_theorem}
	Let the step size $\tau$ be smaller than the threshold $\bar{\tau}$ in \eqref{tau_threshold}, where $L_F$ and $L_H$ are determined in \cref{assumption} for given $\gamma$. If $\gamma^2$ is a polynomial, and the parameters satisfy $ a\geqslant\max_l \|\nabla^2_{\bm{x}}\zeta_{_{l}}\|_2 $ with $\zeta_{_{l}}$'s defined in \eqref{h2represent_p} and $ b\geqslant L_f $, then the sequence generated by the DYS algorithm associated with the splitting \eqref{splitting form} converges to a critical point of \eqref{discrete problem}.
\end{theorem}


\section{Numerical simulations}\label{section4}
In this section, we apply the DYS algorithm to compute the ESC for different anisotropic surface energy densities.
If not explicitly specified, the computational domain is chosen as $\Omega=(0,1)^d$, where $d=2,3$. The mesh size is set to be $h=1/256$. Additionally, we initialize the step size as $\tau_0=1$, and choose the splitting parameters $a=10$ and $b=2$. The parameters for updating the step size are selected as $c_0=1$ and $c_1=10$. The initial condition is represented by a circle centered at $(0.5,0.5)$, as depicted in Figure \ref{strongly_epsilon}(a). Specifically, it is defined by 
\begin{equation}\label{initial_circle}
	\phi(x,y,t=0)=-\tanh\left(\frac{\sqrt{
 (x-0.5)^2+(y-0.5)^2}-0.3}{2\sqrt{2}\times10^{-2}}\right).
\end{equation}
This function is discretized as in Section 2 to produce an initial guess for $\bm{y}^{(0)}$.
As the output of Algorithm \ref{alg:algorithm3}, $\bm{z}^{(n+1)}$ approximates the ESC after reshaping  into a matrix representation. 
The iterations are terminated when the condition $\|\bm{y}^{(n+1)}-\bm{z}^{(n+1)}\|/\tau<10^{-8}$ is satisfied during the course of our simulations.

\subsection{Strongly anisotropic cases for four-fold anisotropy}
The equilibrium shapes exhibit pyramidal structures with sharp corners when the anisotropy strength $\alpha>\frac{1}{15}$ for four-fold anisotropy \eqref{four-fold}, which presents numerical challenges in simulations. We will show that this difficulty can be well circumvented using the proposed method in our comprehensive numerical investigation.

\subsubsection{Convergence for diminishing $\varepsilon$} 
We consider $\varepsilon=0.08, 0.04, 0.02, 0.01$, and fix the anisotropy strength $ \alpha=0.2 $.
The corresponding zero level-sets of $\bm{z}^{(n+1)}$ representing the ESC are depicted in Figure \ref{strongly_epsilon}(b). 
Apparently, sharp corners are well captured in the numerical results and the four ``facets'' are clearly observed for this four-fold symmetry, consistent with the theoretical solution obtained via the Wulff construction \cite{wulff1901}. A zoom-in plot reveals the monotone convergence of the numerical solution towards the exact one as $\varepsilon$ decreases.
\begin{figure}[htpb]
	\centering
	\subfloat{\includegraphics[trim=0.8cm 0cm 1.5cm 0.5cm, clip,width=0.42\linewidth]{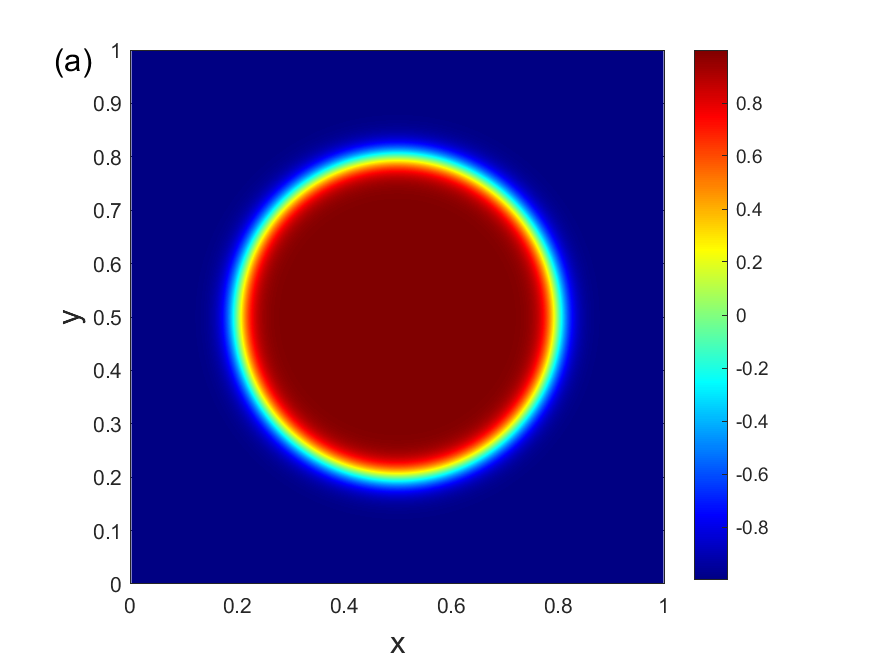}}
	\subfloat{\includegraphics[trim=1.5cm 0cm 1cm 0.5cm, clip,width=0.42\linewidth]{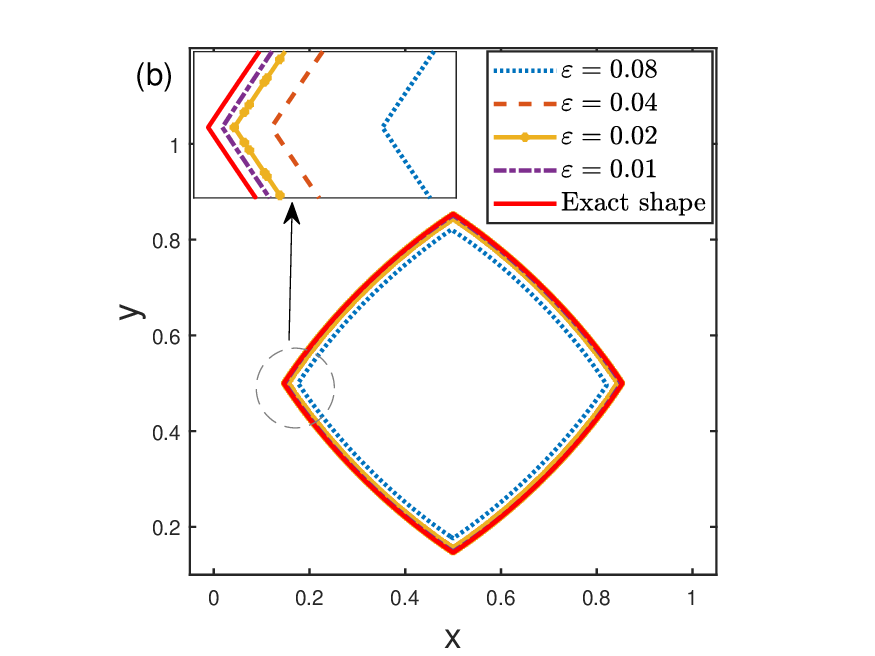}}	
	\caption{Numerical convergence for diminishing $\varepsilon$. (a) shows that the initial shape is a circle. In (b), the zero level-sets for different thickness parameters ranging from $ \varepsilon=0.08$ (blue dotted line) to $\varepsilon=0.01 $ (purple dashed-dot line) are plotted in the strongly anisotropic case with $\alpha=0.2$. These numerical results are compared with the theoretical solution obtained via the Wulff construction (red solid line), where the convergence trend is clearly observed, and sharp corners are well captured.\label{strongly_epsilon}}
\end{figure}

To quantify the convergence rate in $\varepsilon$, we employ the manifold distance metric introduced in \cite{zhao2021energy} as an error measure. Let $\Omega_1$ and $\Omega_2$ be the regions enclosed by the simple closed curves $\Gamma_1$ and $\Gamma_2$, respectively. The manifold distance $M(\Gamma_1, \Gamma_2)$ is defined as the area of the symmetric difference between $\Omega_1$ and $\Omega_2$:
\begin{equation*}
	M(\Gamma_1, \Gamma_2) = \left|(\Omega_1 \backslash \Omega_2) \cup (\Omega_2 \backslash \Omega_1)\right| = \left|\Omega_1\right| + \left|\Omega_2\right| - 2\left|\Omega_1 \cap \Omega_2\right|,
\end{equation*}
where $|\Omega|$ denotes the area of $\Omega$. In Table \ref{manifold_distance}, the manifold distance is evaluated as $\varepsilon$ gradually decreases from $0.08$ to $0.01$. It is observed that the convergence rate in $\varepsilon$ is first-order. 
\begin{table}[htpb]
    \centering
    \begin{tabular}{c|c|c}
			\hline
			$\varepsilon$&Distance&Order\cr\hline
			0.08&$4.57\times10^{-2}$&/\cr
			0.04&$1.51\times10^{-2}$&1.59\cr
			0.02&$4.82\times10^{-3}$&1.64\cr
			0.01&$1.79\times10^{-3}$&1.43\cr
			\hline
		\end{tabular}
    \caption{Manifold distance between the equilibrium shapes obtained by the proposed method and the exact shape under different thickness parameters $\varepsilon$.}
    \label{manifold_distance}
\end{table}

\subsubsection{Comparisons with $H^{-1}$ gradient flow}
We make a systematic comparison of our proposed method with the $H^{-1}$ gradient flow approach in the simulation of the ESC. 

For the gradient flow approach, we solve the anisotropic Cahn-Hilliard equation \eqref{anisotropic_cahn} with a bi-harmonic regularization by modifying the chemical potential as $\mu=f'(\phi)-\varepsilon^2 \nabla \cdot\bm{m}+\beta\varepsilon^2\Delta^2\phi$,
where $\beta$ controls the regularization strength. This regularization is necessary for tackling the ill-posedness of the gradient flow dynamics of strongly anisotropic surface energy \cite{wise2007solving}. The regularized anisotropic Cahn-Hilliard equation is numerically solved using the convex splitting method \cite{cheng2020weakly}. The parameter are fixed at $\varepsilon=0.02$ and $\beta=10^{-4}$. The time step size is $10^{-3}$, and the termination criterion is set to be $\|\phi^{(n+1)}-\phi^{(n)}\|\leqslant10^{-8}$. 

We first show a visual comparison of the ESC computed by the proposed method and the gradient flow approach under four different strengths of anisotropy $\alpha=0.1\sim0.4$. 
As shown in Figure \ref{strength}, the four-fold pyramidal shapes are observed for both methods and the computed ESCs look very similar.
However, a closer look at their corners suggests significant differences. 
For the proposed method, sharp corners always appear for all strong anisotropy strengths $\alpha$, and the corners become sharper and more pronounced as $\alpha$ increases. 
In contrast, in the numerical results obtained by the gradient flow approach, the corners are smooth and round-off which does not align with the theoretical prediction.
\begin{figure}[htpb]
	\centering
	\subfloat{\includegraphics[trim=1.5cm 0cm 1cm 0.5cm, clip,width=0.42\linewidth]{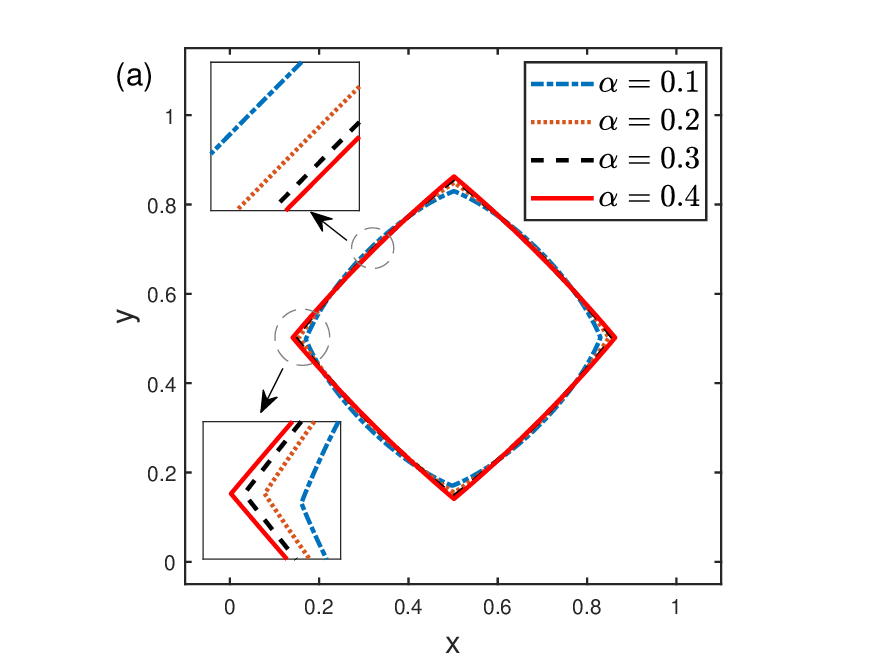}}	
	\quad
	\subfloat{\includegraphics[trim=1.5cm 0cm 1cm 0.5cm, clip,width=0.42\linewidth]{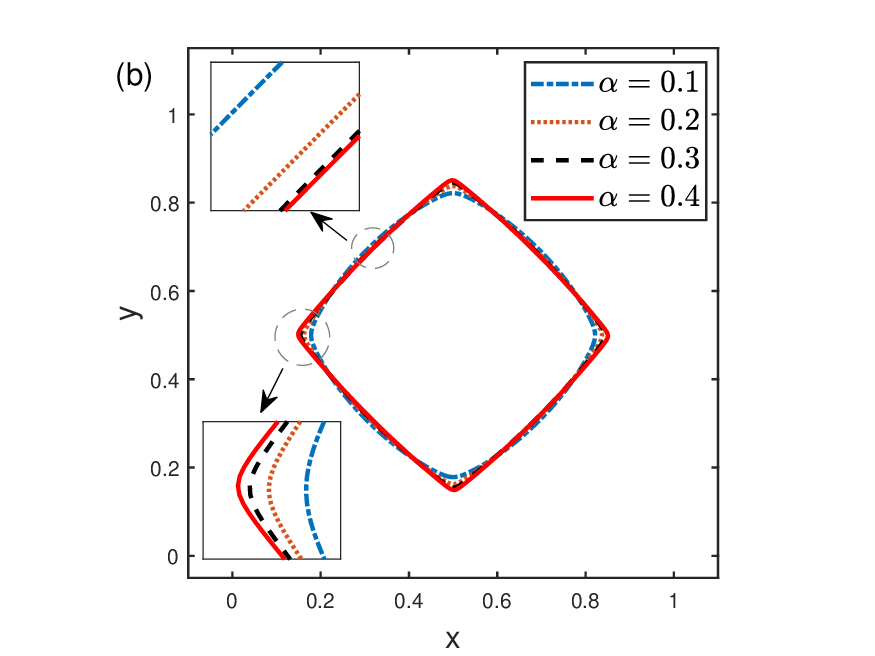}}	
	\caption{\label{strength} Comparison of the equilibrium zero level-sets obtained by the proposed method (shown in (a)) and the $H^{-1}$ gradient flow approach (shown in (b)) for four different anisotropy strengths $\alpha$. The zoom-in plots indicate that the corners of the pyramids are always sharp for the proposed method, while they are round-off for the gradient flow approach.}
\end{figure}

To give a more quantitative investigation, we focus on the case with $\alpha=0.2$. 
We show the ESCs of the proposed method, the gradient flow approach, and the theoretic solution in Figure \ref{compare}(a). 
It is observed that both methods provide promising numerical approximations to the theoretical solution and the computed contours almost agree with each other except at the corners. 
The different performance of the two methods can be further analyzed by computing the orientations of the contours. We represent the orientation of a contour curve by the angle between its tangent vector and the y-axis. Then the behavior of a contour curve is clearly revealed by the orientation plot versus the arc length (starting at the left corner point), as depicted in Figure \ref{compare}(b). 
It is evident that the numerical solution obtained by the proposed method agrees well with the exact one in terms of orientation, and there is a significant change in orientation near the left corner, indicating the presence of a sharp corner. 
In contrast, the gradient flow approach gives relatively smooth results with a gradually varying orientation near the left corner. Although a diminishing $\beta$ can enhance the transition in the orientation and improve the approximation to the exact solution, there is still a much larger error nearby the sharp corner in comparison with our proposed method.

\begin{figure}[htpb]
	\centering
	\subfloat{ \includegraphics[trim=1.5cm 0cm 1cm 0.5cm, clip,width=0.40\linewidth]{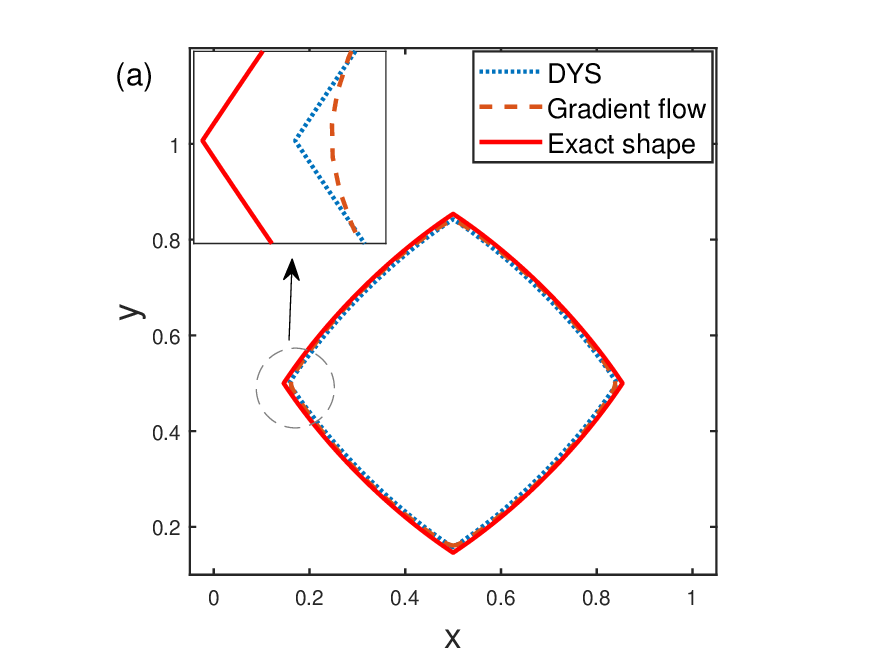}}	
	\quad
	\subfloat{\includegraphics[trim=1cm 0.0cm 1.5cm 0.5cm, clip,width=0.55\linewidth]{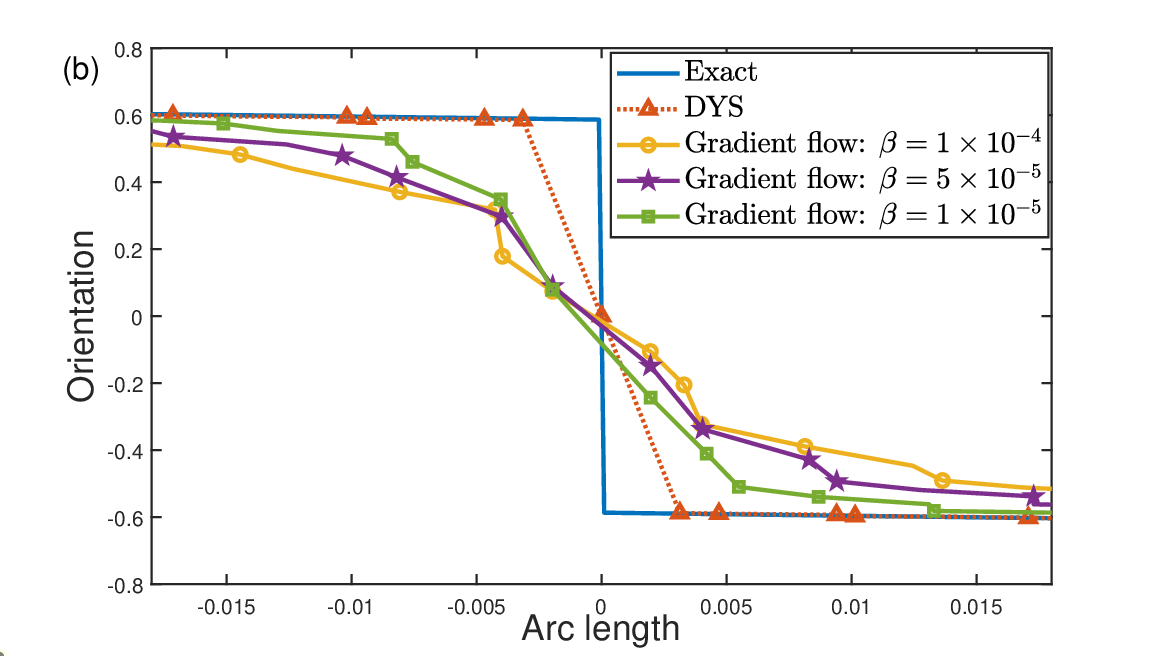}}	
	\caption{\label{compare}Comparison of the ESCs obtained by the proposed method, the $H^{-1}$ gradient flow approach, and the theoretic result. The contour plots are given in (a), while the orientation plots near the left corners of the ESCs are shown in (b). The gradient flow results with different regularization strengths $\beta$ are provided to show the convergence tendency as $\beta\rightarrow0$.}
\end{figure}


As an important indicator for optimization problems, the extent to which the optimality condition is violated is usually checked. We compute the $L_2$ norm of the residual of the optimality condition~\cite{andreani2008augmented}
\[
\|\operatorname{P}_{Box[\mathbf{-1,1}]}(\bm{z}^n-\nabla F(\bm{z}^n)-\nabla H(\bm{z}^n)-\lambda \mathbf{1})-\bm{z}^n\|_2,
\]
with $\lambda$ being a constant minimizing this norm,
and plot it versus the CPU time cost in Figure~\ref{cputime}(a) for the above two approaches. It is remarkable that the residual of the optimality condition is seriously reduced with time for the proposed method, leading to a much milder violation of the optimality condition than the gradient flow approach does. 
The residual of the optimality condition for the gradient flow approach is confined at the order of $\mathcal{O}(10^{-4})$, which may be attributed to the violation of the box constraint. 


Finally, we also test the original energy decay of these two methods. We set the initial step size $\tau_0=0.1$ for the proposed method, and depict the evolution of the original energy function with the CPU time in Figure~\ref{cputime}(b). 
Both methods exhibit energy decaying behavior. It is noteworthy that the ESC computed by the proposed method enjoys a lower original energy than that of the gradient flow approach.
\begin{figure}[htpb]
	\centering
	\subfloat{\includegraphics[trim=0.2cm 0cm 0.2cm 0.5cm, clip,width=0.42\linewidth]{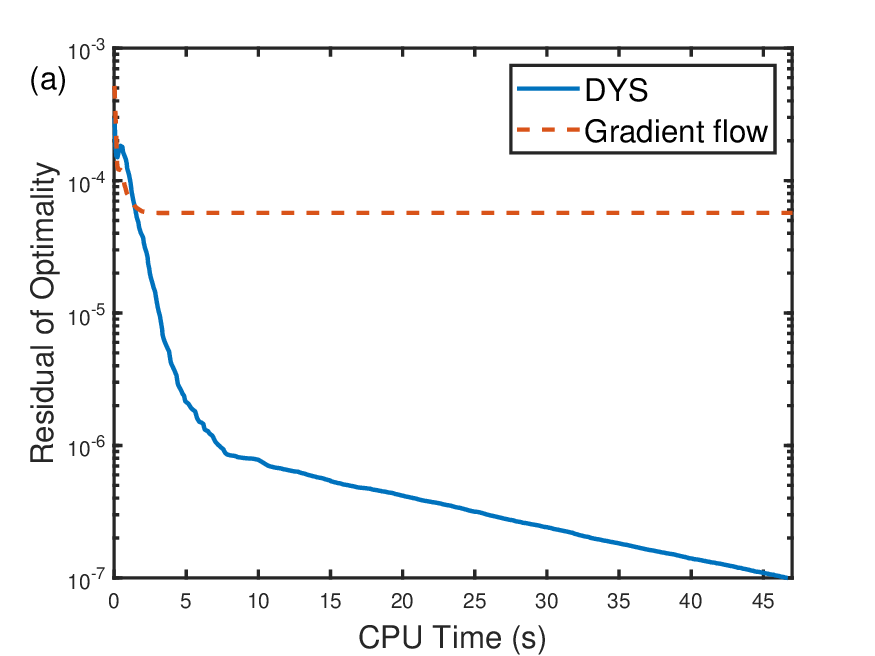}}
 \quad
	\subfloat{\includegraphics[trim=0.2cm 0cm 0.2cm 0.5cm, clip,width=0.42\linewidth]{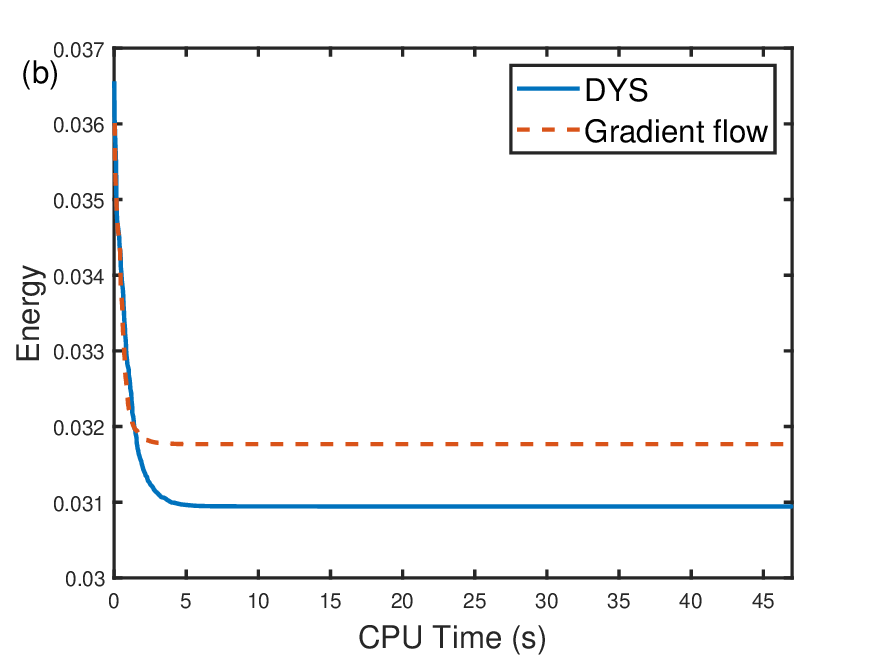}}	
	\caption{The residual of the optimality condition (shown in (a)) and the original energy (shown in (b)) versus the CPU time. The proposed method achieves a smaller residual of the optimality condition at about $10^{-7}$ while the gradient flow approach remains confined to a scale around $10^{-4}$. Moreover, the proposed method favors lower original energy.}
        \label{cputime}
\end{figure}





\subsubsection{Mass-conservation and bound-preservation}

As an encouraging feature of the proposed method, the mass-conservation and the box constraint can be naturally guaranteed during the iteration due to the projection in the three-operator-splitting.
As shown in Figure \ref{fig_bound_mass}(a), the bounds of $ \bm{z}^{(n+1)} $ are well preserved within the range $ [-1,1] $. Moreover, we also compute the relative loss in the discrete total mass of $ \bm{z}^{(n+1)} $, namely $ \frac{\bm{1}^T\bm{z}^{(n+1)}}{\bm{1}^T\bm{z}^{(1)}} -1 $.
As shown in Figure \ref{fig_bound_mass}(b), the relative loss is very close to zero within the machine precision at the magnitude of $ \mathcal{O}(10^{-15}) $. 
These results numerically validate the mass-conservation and bound-preserving properties of the proposed method.
\begin{figure}[htpb]
	\centering
	\subfloat{\includegraphics[trim=0.3cm 0cm 0.5cm 0.0cm, clip,width=0.42\linewidth]{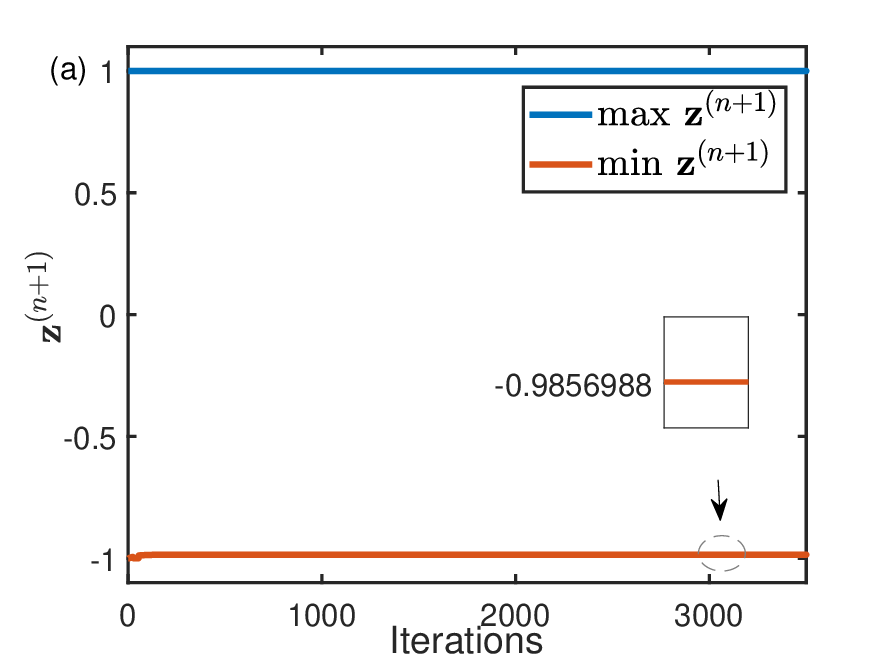}}	
	\quad
	\subfloat{\includegraphics[trim=0.3cm 0cm 0.5cm 0.0cm, clip,width=0.42\linewidth]{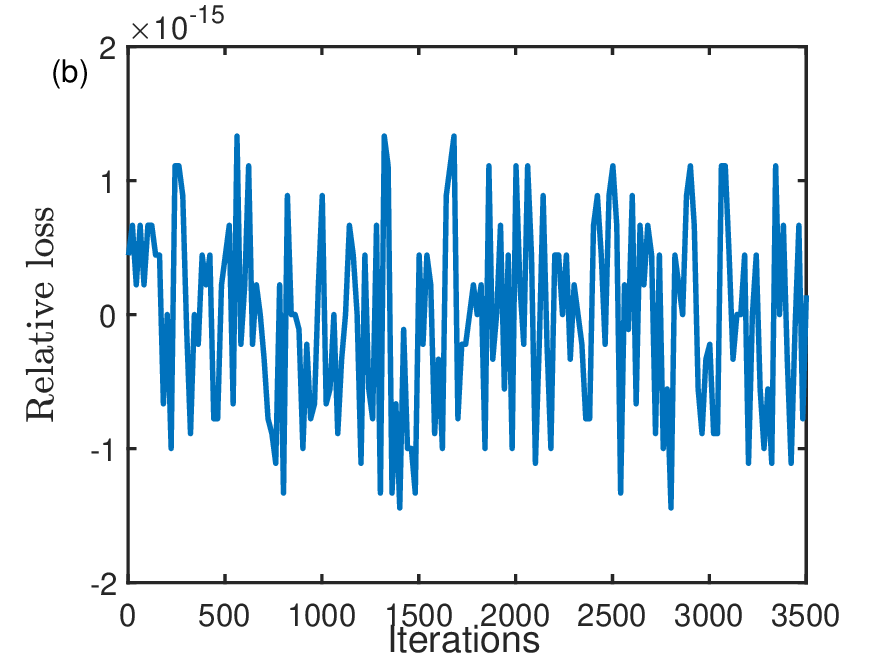}}	
	\caption{\label{fig_bound_mass}Mass-conservation and bound preservation. (a) shows the upper and lower bounds of $\bm{z}^{(n+1)}$. (b) shows the relative loss in the discrete mass of $\bm{z}^{(n+1)}$.}
\end{figure}

\subsubsection{Different initial profiles}
We investigate the effects of different initial profiles. We first consider an initial condition given by two circles with different sizes: 
\begin{align*}
	\phi(x,y,t=0)=&-\tanh\left(\frac{\sqrt{(x-0.4)^2+(y-0.6)^2}-0.25}{\varepsilon}\right)\\
	&-\tanh\left(\frac{\sqrt{(x-0.8)^2+(y-0.2)^2}-0.16}{\varepsilon}\right)+1.
\end{align*}
Figure \ref{fig_two_circle_energy} illustrates the energy decaying effect under the parameters $\alpha=0.2$ and $ \tau_0=0.1 $ during the iterations. It is evident that both the original energy and the modified energy \eqref{new_energy} exhibit two distinct rapid decreases. The first significant decreases occur immediately when the iteration begins, and the two circles evolve rapidly to two four-fold pyramids induced by the anisotropy. Then it follows that these two pyramids merge into one, during which the second rapid decrease in the energies takes place.
Notably, similar effects are also observed in a different simulation framework via $H^{-1}$ gradient flow approach \cite{shen2018stabilized,chen2019fast}.


\begin{figure}[htpb]
	\centering
	\includegraphics[trim=0cm 0.1cm 0.5cm 0.5cm, clip,width=0.42\linewidth]{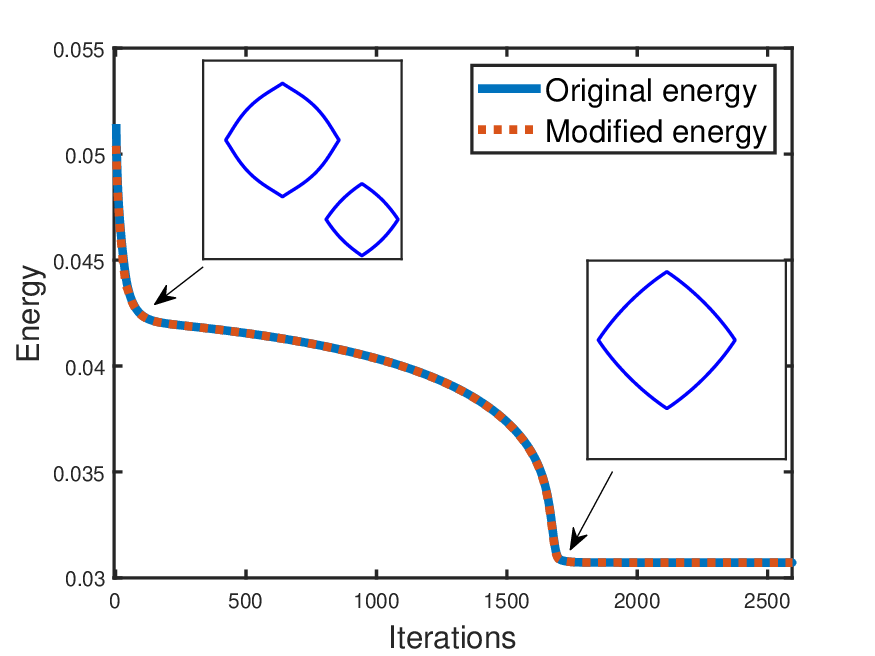}
	\caption{\label{fig_two_circle_energy}The evolution of energy functions when the initial profile is two-circle. Both the original and modified energies decay with iteration. The inset plots indicate two particular snapshots of crystal profiles around which the energies behave quite differently. 
 }
\end{figure}

As a second example, we consider the random initial concentration field given by
\[
\phi(x,y,t=0)=-0.5+0.001\operatorname{rand}(x,y),
\]
whose profile is displayed in Figure \ref{fig_rand}(a). We choose $\alpha=0.2$, $\varepsilon=0.02$, and the splitting parameter $a=100$. As depicted in Figure \ref{fig_rand}(b), a four-fold pyramid is observed as the ESC as a result of periodic boundary conditions, indicating the phase separation effect of the phase-field model.
\begin{figure}[htpb]
	\centering
	\subfloat{\includegraphics[trim=1cm 0cm 1cm 0.5cm, clip,width=0.40\linewidth]{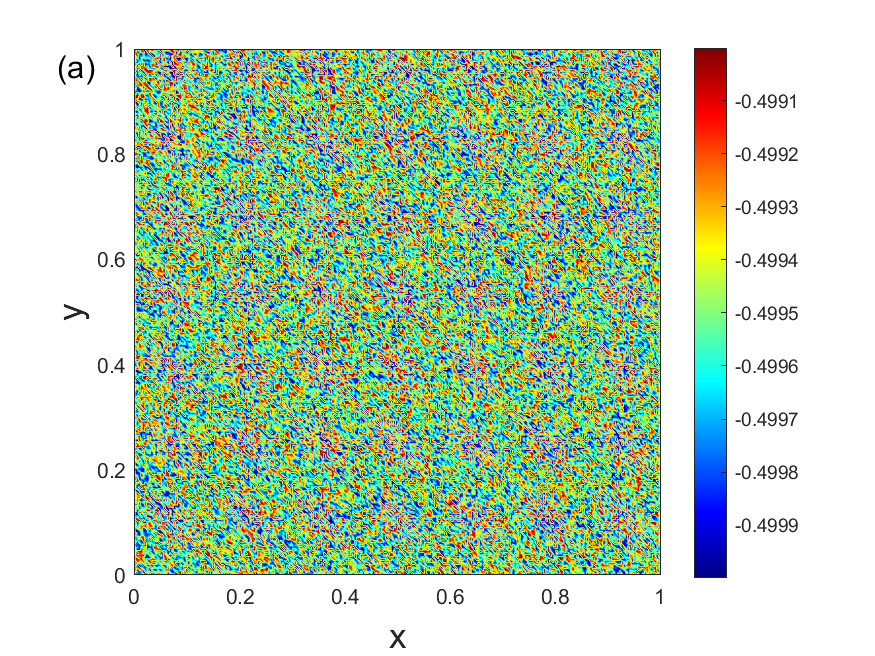}}
    \quad
	\subfloat{\includegraphics[trim=1cm 0cm 1cm 0.5cm, clip,width=0.40\linewidth]{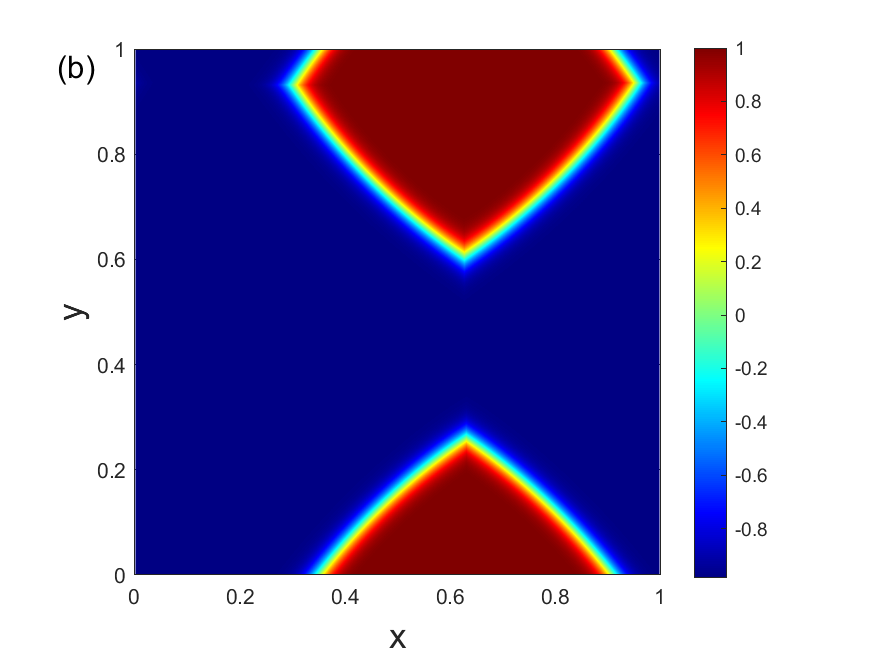}}	
	\caption{\label{fig_rand} The ESC (shown in (b)) for random initial profile (shown in (a)).}
\end{figure}

\subsection{Anisotropy with different symmetries}
In two-dimensional cases, the anisotropy function $\gamma(\bm{n})=\gamma(\theta)$  can be represented as a periodic function of $\theta$ with $\tan(\theta)=n_2/n_1$.  The commonly used $k$-fold smooth anisotropy function is
\[
\gamma(\theta)=1+\alpha\cos(k\theta), \quad\theta\in(-\pi,\pi],
\]
where  $\alpha$ controls the strength of anisotropy, and $k$ denotes the symmetry parameter. For instance, when $k=4$, it corresponds to the two-dimensional case of the four-fold anisotropy \eqref{four-fold} whose ESC is invariant under a $\frac{\pi}{2}$-rotation. In this subsection, we numerically study the ESCs of anisotropic energy functionals with different symmetries using our proposed method. Specifically, we consider $k=2,~3,~6$. 

When $\alpha>\frac{1}{k^2-1}$, the systems exhibit strong anisotropy \cite{wise2007solving}. In our numerical simulations, we set $\alpha=0.4$ and illustrate the corresponding ESCs for different $ k $ in Figure \ref{fig_k_fold}. Notably, the computed ESCs with correct numbers of ``facets" and sharp corners are obtained. 
\begin{figure}[htpb]
	\centering
	\subfloat[$ k=2 $]{\includegraphics[trim=1cm 0cm 1cm 0.75cm, clip,width=0.333\linewidth]{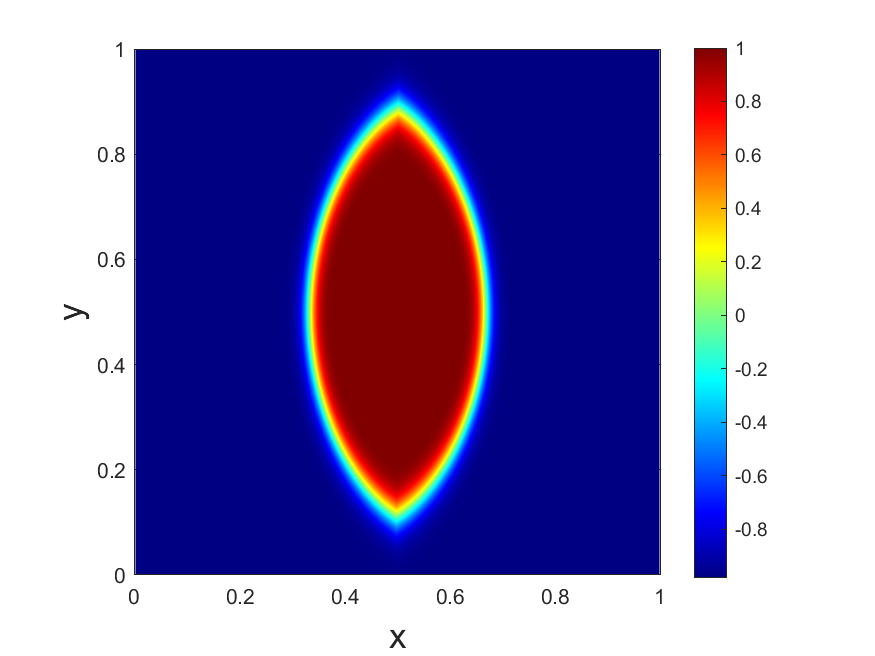}}
	\subfloat[$ k=3 $]{\includegraphics[trim=1cm 0cm 1cm 0.75cm, clip,width=0.333\linewidth]{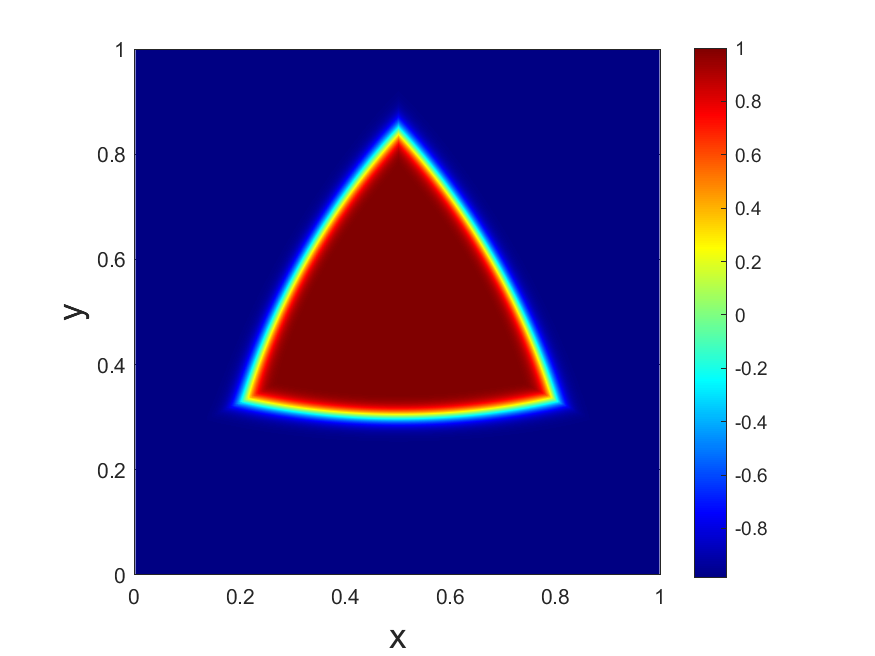}}	
	\subfloat[$ k=6 $]{\includegraphics[trim=1cm 0cm 1cm 0.75cm, clip,width=0.333\linewidth]{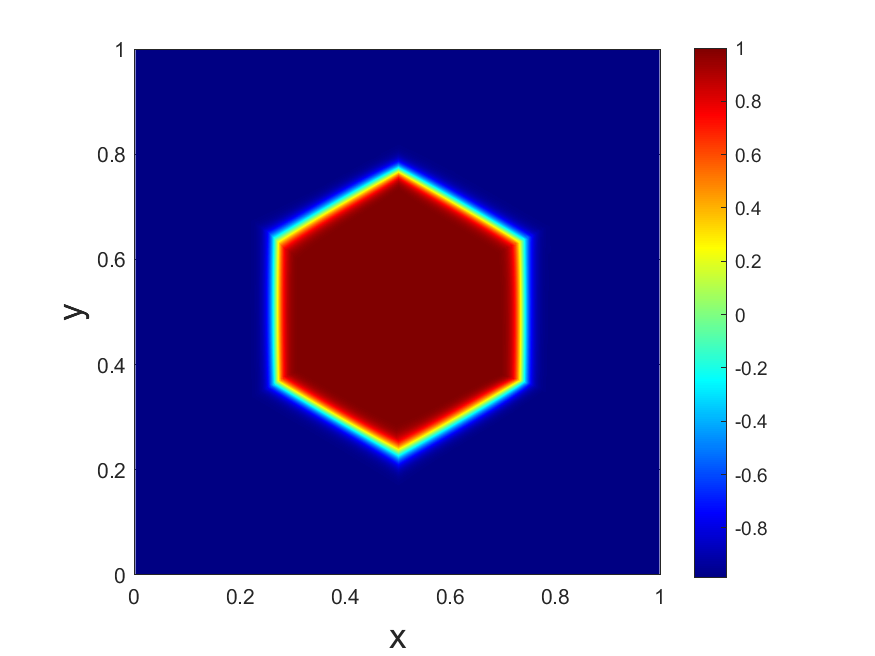}}
	\caption{\label{fig_k_fold}The ESCs of anisotropic energy functionals with different symmetry parameters $ k=2,~3,~6$.}
\end{figure}

\subsection{Anisotropy of Riemannian metric form}
Besides the $k$-fold type anisotropy, there is another commonly-used anisotropy, namely the anisotropy of Riemannian metric form, which is defined as follows \cite{deckelnick2005computation,jiang2019sharp}:
\begin{equation}\label{riemann_euq}
	\gamma(\bm{n})=\sum_{k=1}^{K}\sqrt{G_k\bm{n}\cdot\bm{n}},\quad G_k=R(-\psi_k)D(\delta_k)R(\psi_k),\quad k=1,2,\ldots,K,
\end{equation}
where the matrices $ D $ and $ R $ are given by
$$
D(\delta)=\left(\begin{array}{cc}
	1 & 0 \\
	0 & \delta^2
\end{array}\right), \quad R(\psi)=\left(\begin{array}{cc}
	\cos \psi & \sin \psi \\
	-\sin \psi & \cos \psi
\end{array}\right).
$$
We consider two sets of parameters:
\begin{enumerate}
	\item $ K=2 $, $ \psi_1=0 $, $ \psi_2=\frac{\pi}{2} $, and $\delta_k=\delta$ for $k=1,~2$;
	
	\item $ K=3 $, $ \psi_1=0 $, $ \psi_2=\frac{\pi}{3} $, $ \psi_3=\frac{2\pi}{3} $, and $\delta_k=\delta$ for $k=1,~2,~3$.
\end{enumerate}
For the first set of parameters, $\gamma(\bm{n})=\sqrt{n_1^2+\delta^2n_2^2}+\sqrt{\delta^2n_1^2+n_2^2}$. It can be viewed as a regularization for the non-smooth anisotropy $\gamma(\bm{n})=|n_1|+|n_2|$ whose corresponding ESC also has sharp corners (arising from the non-smoothness of $\gamma(\bm{n})$). 
The parameter $\delta$ plays a role of penalizing the sharp corners. When $ \delta $ decreases to zero, the ESC will exhibit ``sharper and sharper'' corners, although its contour is still smooth for non-zero $\delta$. 

By employing our proposed method, we conduct numerical investigations for a sequence of decreasing regularization parameters $\delta^2=0.01,~0.001,~0.0001$ with the corresponding splitting parameters $a=10,~50,~100$. As depicted in Figure \ref{riemann_1}, the ESCs of square and hexagon shapes are respectively observed for the two sets of parameters. Apparently, the corners become ``sharper'' as $\delta$ decreases. Interestingly, the facets of the ESCs are flatter than that of the smooth $k$-fold anisotropy, making them look more like standard squares and hexagons. 
These findings are in agreement with the observations reported in \cite{jiang2019sharp}.

\begin{figure}[htpb]
	\centering
	\subfloat{\includegraphics[trim=0.5cm 0cm 0.5cm 0.5cm, clip,width=0.42\linewidth]{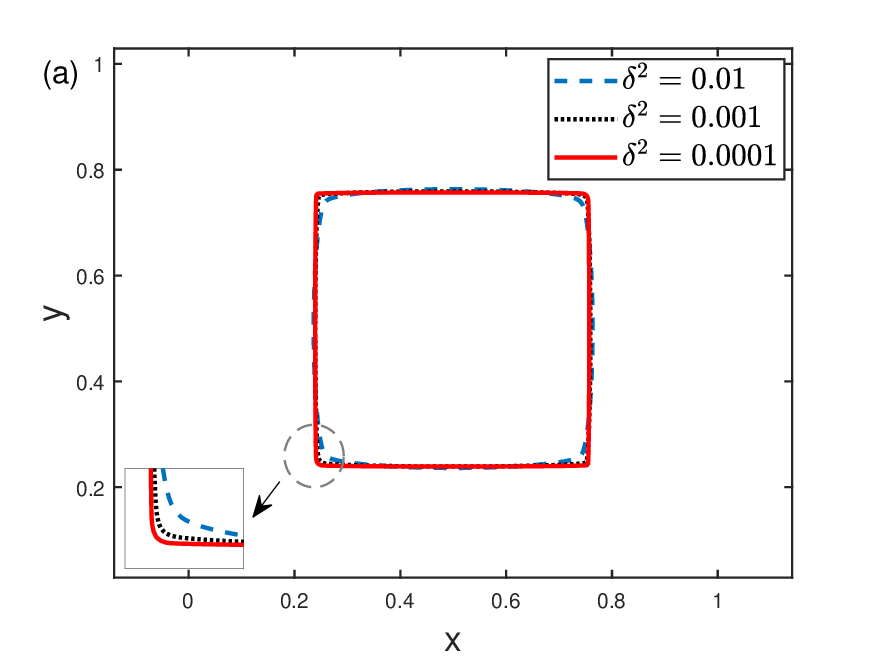}}
	\subfloat{\includegraphics[trim=0.5cm 0cm 0.5cm 0.5cm, clip,width=0.42\linewidth]{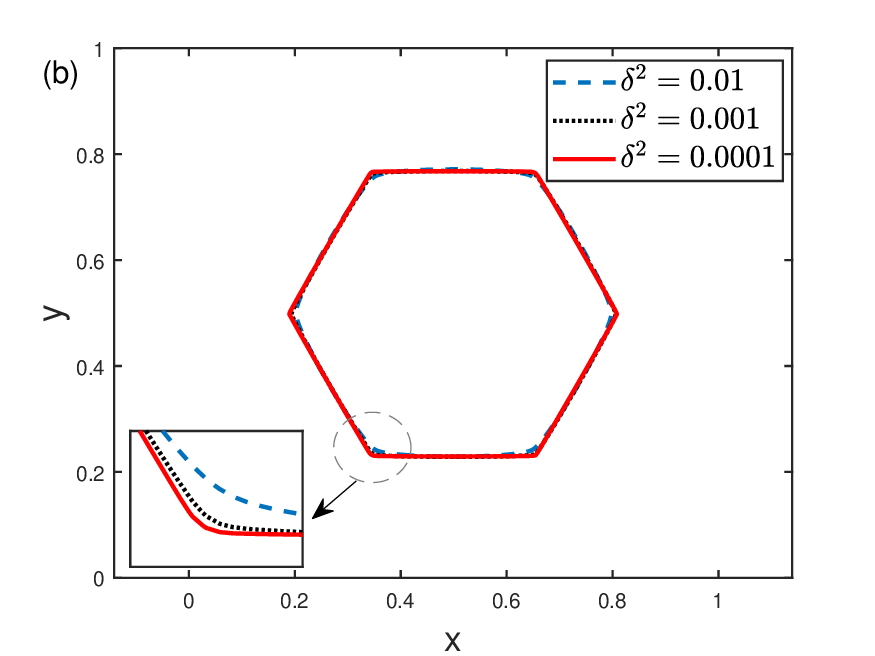}}	
	\caption{\label{riemann_1}The ESCs of anisotropic functionals of Riemannian metric form. The parameters used in simulations are: (i)	$ K=2 $, $ \psi_1=0 $,$ \psi_2=\frac{\pi}{2} $ (shown in (a)); (ii) $ K=3 $, $ \psi_1=0 $, $ \psi_2=\frac{\pi}{3} $, $ \psi_3=\frac{2\pi}{3} $ (shown in (b)). Zoom-in plots nearby the corners are also shown to demonstrate the regularization effect for decreasing $\delta$.}
\end{figure}

\subsection{Three-dimensional simulations}
In three-dimensional cases, we conduct numerical simulations for two types of anisotropy. 

For the first type, we consider an ellipsoidal anisotropy of the form $\gamma(\bm{n})=\sqrt{2n_1^2+n_2^2+n_3^2}$. The theoretical prediction for the ESC under this anisotropy is a self-similar ellipsoid given by $\frac{x^2}{2}+y^2+z^2=1$ \cite{zhao2020parametric}. In our simulations, we initialize with a ball-shaped concentration field:
\[
\phi(x,y,z,t=0)=-\tanh\left(\frac{\sqrt{(x-0.5)^2+(y-0.5)^2+(z-0.5)^2}-0.3}{\sqrt{2}\varepsilon}\right).
\]
The computed ESC indeed exhibits an ellipsoidal shape, as shown in Figure \ref{fig_ellipsoid}(a).

Finally, we study the four-fold anisotropy \eqref{four-fold} with $\alpha=0.2$ and $\varepsilon=0.02$ as the second type of anisotropy. The initial condition is also set to be the ball-shaped concentration field. The computed ESC under this anisotropy is a double-sided pyramid with sharp corners, as depicted in Figure \ref{fig_ellipsoid}(b).

\begin{figure}[htpb]
	\centering
	\subfloat{\includegraphics[width=0.42\linewidth]{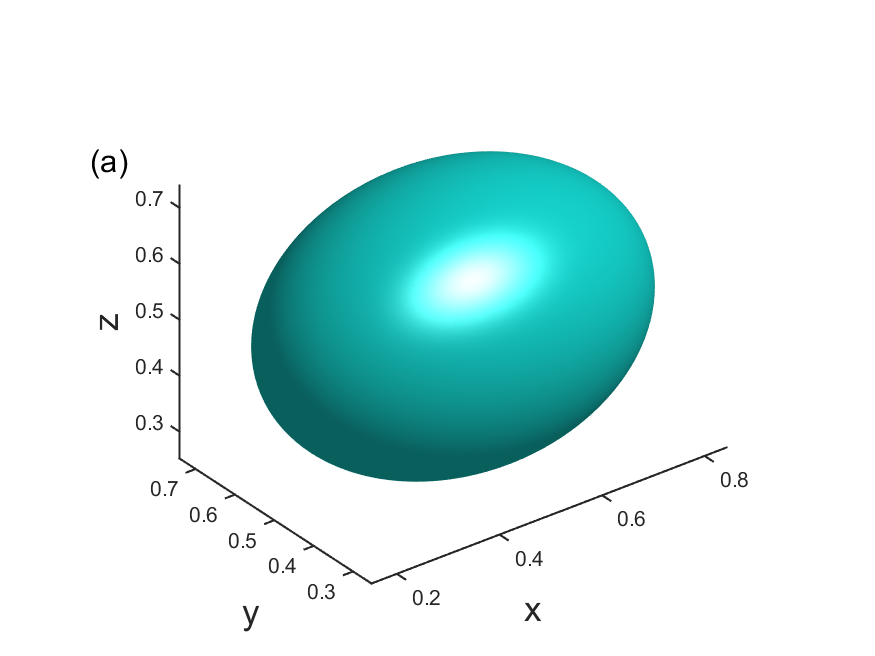}}
    \quad
	\subfloat{\includegraphics[width=0.42\linewidth]{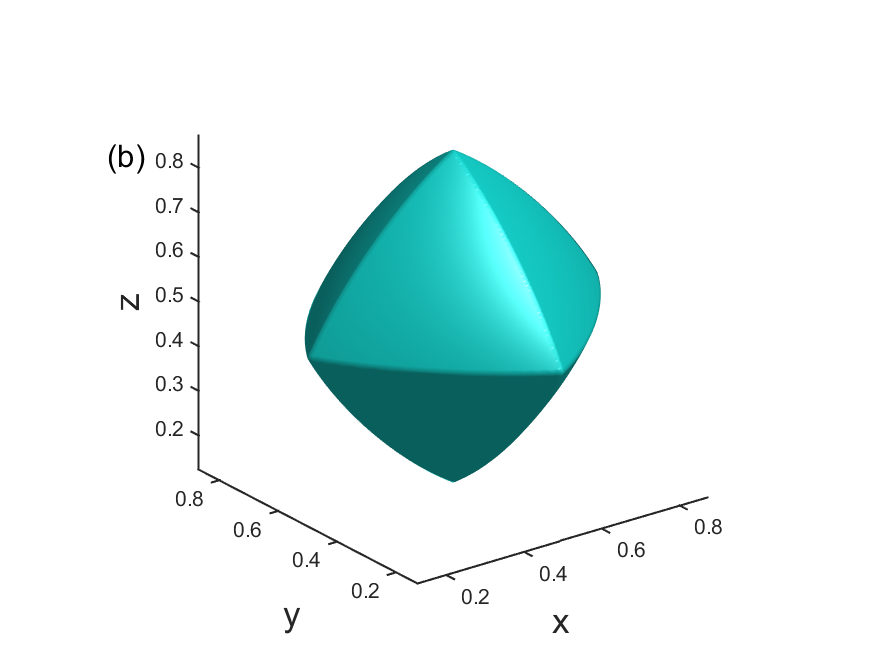}}	
	\caption{\label{fig_ellipsoid}The ESCs for ellipsoid anisotropy (shown in (a)) and four-fold anisotropy (shown in (b)).}
\end{figure}


\section{Conclusion and discussion}
In this paper, we proposed a novel numerical method based on the DYS algorithm for determining the ESCs. By splitting the discretized anisotropic surface energy function into a convex function, a concave function and an indicator function of a bounded convex set, we solved the energy minimization problem under the mass-conservation and the box constraint using the DYS algorithm.
The proposed method has the advantage of conserving the total mass and preserving the bound automatically at each iteration. This algorithm was efficiently implemented using fast solvers. More importantly, we proved that the proposed method has global convergence to some critical points. In comparison with the $H^{-1}$ gradient flow approach for minimizing anisotropic energy functionals with regularizations, the computed ESCs based on our proposed method exhibit sharp corners, which is consistent with the theoretical prediction and yields better accuracy. 

%


We employed the proposed method to successfully simulate the ESCs for $k$-fold anisotropic surface energies and Riemannian metric types of anisotropy. The desired properties, including mass-conservation, bound-preservation, energy decay and convergence to theoretical solutions, were numerically observed. In addition, the ESC computed by the proposed method shows a lower residual of the optimality condition and a lower original energy than that by the gradient flow approach does.

Despite the encouraging performance in simulating the ESC, the proposed method cannot be directly applied to predict evolution dynamics of crystals. This motivates us to investigate the connection between the evolution dynamics of the proposed method and the gradient flow dynamics. 
Furthermore, for more complex anisotropic energy functionals, the equilibrium states can be more complicated and more ESCs may exist. This is quite typical in solid-state dewetting problems where boundary energies are included. 
We will study the performance of operator-splitting optimization approaches in predicting the equilibrium states of solid-state dewetting problems~\cite{Jiang17}.


\appendix
\section{Semi-algebraic functions}\label{appe_semi}
Semi-algebraic function plays an important role in the convergence study of optimization algorithms. This appendix introduces its definition and provides some typical examples (c.f. \cite{bolte2014proximal}).

\begin{definition}[Semi-algebraic sets and functions]\label{def_semi}
	\quad
	\begin{enumerate}
		\item	 A subset $S$ of $\mathbb{R}^n$ is a real semi-algebraic set if there exists a finite number of real polynomial functions $f_{i j}, g_{i j}: \mathbb{R}^n \rightarrow \mathbb{R}$ such that
            \[
                S=\bigcup_{j=1}^p \bigcap_{i=1}^q\left\{x \in \mathbb{R}^n: f_{i j}(x)=0 \text { and } g_{i j}(x)<0\right\} .
            \]
		
		\item	A function $f: \mathbb{R}^n \rightarrow(-\infty,+\infty]$ is called semi-algebraic if its graph
		$$\left\{(x, t) \in \mathbb{R}^{n+1}: f(x)=t\right\}$$
		is a semi-algebraic subset of $\mathbb{R}^{n+1}$.
	\end{enumerate} 
\end{definition}

\begin{property}\label{example}
	The following functions are the semi-algebraic functions:
	\begin{enumerate}
		\item Real polynomial functions;
  
		\item Indicator functions of semi-algebraic sets;
		
		\item Finite sums and product of semi-algebraic functions;
		
		\item Composition of semi-algebraic functions.
	\end{enumerate}
\end{property}

\section{Proof of \cref{lemma_boundness}}\label{appe_lemma_bound}
	
	Because $F$ has a Lipschitz continuous gradient, for $\tilde{\bm{x}}=\bm{x}-\frac{1}{L_{_F}} \nabla F(\bm{x})$ we have 
	\begin{equation*}
		\begin{aligned}
            F(\tilde{\bm{x}})\leqslant F(\bm{x})+\left\langle\nabla F(\bm{x}),\tilde{\bm{x}}-\bm{x}\right\rangle+\frac{L_{_{F}}}{2}\left\|\tilde{\bm{x}}-\bm{x}\right\|^2 = F(\bm{x})-\frac{1}{2 L_{_{F}}}\|\nabla F(\bm{x})\|^2.
		\end{aligned}
	\end{equation*}
        This inequality together with the lower boundedness of $\nabla F$ implies there exists a constant $\zeta^*>-\infty$ such that
        \begin{equation}\label{f_bounded_below}
            F(\bm{x})-\frac{1}{2 L_{_{F}}}\|\nabla F(\bm{x})\|^2\geqslant \zeta^*.
        \end{equation}
	From \eqref{step3_alg1} and the first-order optimality condition \eqref{first_optimal}, we have
	\begin{equation}\label{norm_first_order}
		\|\bm{x}^{(n)}-\bm{z}^{(n)}\|^2=\|x^{(n-1)}-\bm{y}^{(n)}\|^2=\tau^2\|\nabla F(\bm{y}^{(n)})\|^2.
	\end{equation}
	Moreover, since $ H $ is a concave function, 
	\begin{equation}\label{h_concave_tangent}
		H(\bm{z}^{(n)})\leqslant H(\bm{y}^{(n)})+\left\langle\nabla H(\bm{y}^{(n)}),\bm{z}^{(n)}-\bm{y}^{(n)}\right\rangle.
	\end{equation}
	Since $ G $ is an indicator function of a bounded set, by  \eqref{step2_alg1}, we know $ \bm{z}^{(n)} $ is bounded and $G(\bm{z}^{(n)})=0.$ 
 Combining  \eqref{f_bounded_below}, \eqref{norm_first_order}, \eqref{h_concave_tangent} and  \eqref{descent_lemma}, we obtain 
	\[
	\begin{aligned}
		&\Theta_\tau\left(\bm{x}^{(1)}, \bm{y}^{(1)}, \bm{z}^{(1)}\right)\geqslant\Theta_\tau\left(\bm{x}^{(n)}, \bm{y}^{(n)}, \bm{z}^{(n)}\right)\\
		=&F(\bm{y}^{(n)})+H(\bm{y}^{(n)})+\frac{1}{\tau}\left\langle \bm{y}^{(n)}-\bm{z}^{(n)},\bm{y}^{(n)}-\bm{x}^{(n)}-\tau\nabla H(\bm{y}^{(n)})\right\rangle\\
		&-\frac{1}{2\tau}\|\bm{y}^{(n)}-\bm{z}^{(n)}\|^2	\\
		=&F(\bm{y}^{(n)})+H(\bm{y}^{(n)})+\frac{1}{2\tau}\|\bm{x}^{(n)}-\bm{y}^{(n)}\|^2-\frac{1}{2\tau}\|\bm{x}^{(n)}-\bm{z}^{(n)}\|^2\\
		&+\left\langle\nabla H(\bm{y}^{(n)}),\bm{z}^{(n)}-\bm{y}^{(n)}\right\rangle\\
		=&(1-\tau L_{_{F}})F(\bm{y}^{(n)})+\left(H(\bm{y}^{(n)})+\left\langle\nabla H(\bm{y}^{(n)}),\bm{z}^{(n)}-\bm{y}^{(n)}\right\rangle\right)+\frac{1}{2\tau}\|\bm{x}^{(n)}-\bm{y}^{(n)}\|^2\\
		&+\tau L_{_{F}}\left(F(\bm{y}^{(n)})-\frac{1}{2 L_{_{F}}}\|\nabla F(\bm{y}^{(n)})\|^2\right)\\
		\geqslant&(1-\tau L_{_{F}})F(\bm{y}^{(n)})+H(\bm{z}^{(n)})+\frac{1}{2\tau}\|\bm{x}^{(n)}-\bm{y}^{(n)}\|^2+\tau L_{_{F}}\zeta^*,
	\end{aligned}
	\]
	where the first equality is a direct rearrangement of $\Theta_\tau$ and the second equality is a result of the identity $2cd-d^2=c^2-(c-d)^2$. 
        Due to the threshold \eqref{tau_threshold}, we have  $ 1-\tau L_{_{F}}>0 $. Since $ \bm{z}^{(n)} $ is bounded and $H$ is continuous, it follows that $ H(\bm{z}^{(n)}) $ is bounded. As a result, both $ F(\bm{y}^{(n)}) $ and $ \bm{x}^{(n)}-\bm{y}^{(n)} $ are bounded, and the coercivity of $ F $ implies the boundedness of $ \bm{y}^{(n)} $. Thus, the sequence $\left\{\left(\bm{x}^{(n)}, \bm{y}^{(n)}, \bm{z}^{(n)}\right)\right\}_{n \geqslant 1}$ generated by \cref{alg:algorithm1} is bounded.

\section{Proof of \cref{lemma_gradient_H}}\label{appe_gradient_h}
We split $H$ into two parts as shown in \eqref{eq_h_1}.
By the chain rule and the definition of $ \bm{p} $ and $ \bm{p}_k $ in \eqref{def_p}, 
\begin{equation*}\label{chain_gradient}
	\nabla h_2(\bm{\phi})=\left[\begin{matrix}
		D_x^T&D_y^T
	\end{matrix}\right]\sum_{k=1}^{m^2}\left[\begin{matrix}
		\bm{e_k} &  \\& \bm{e_k} 
	\end{matrix}\right]\nabla_{\bm{p}_k} h_2.
\end{equation*}
Moreover, by the definition of $\bm{n}_k$ in \eqref{def_nk}
\begin{equation*}\label{chain_h2}
	\nabla_{\bm{p}_k} h_2=\varepsilon^2\left(
	(\gamma^2(\bm{n}_k)-a)\bm{p}_k+\gamma(\bm{n}_k)|\bm{p}_k|(I_2-\bm{n}_k\otimes\bm{n}_k)\nabla_{\bm{n}}\gamma(\bm{n}_k)
	\right)=\varepsilon^2\left[\begin{matrix}
		\mathcal{A}_1(\bm{p}_k)\\
		\mathcal{A}_2(\bm{p}_k)
	\end{matrix}\right].
\end{equation*}
With these two equalities, we arrive at the conclusion after a direct calculation:
\begin{equation}\label{grad_h2}
	\begin{aligned}
		\nabla H(\bm{\phi})
		&=\left[
		\begin{matrix}
			f'(\phi_1)-b\phi_1 \\f'(\phi_2)-b\phi_2 \\ \cdots\\f'(\phi_{m^2}) -b\phi_{m^2}
		\end{matrix}
		\right]+\varepsilon^2\left[\begin{matrix}
			D_x^T&D_y^T
		\end{matrix}\right]\sum_{k=1}^{m^2}\left[\begin{matrix}
			\bm{e_k} &  \\& \bm{e_k} 
		\end{matrix}\right]
		\left[\begin{matrix}
			\mathcal{A}_1(\bm{p}_k)\\
			\mathcal{A}_2(\bm{p}_k)
		\end{matrix}\right]\\
		&=\left[
		\begin{matrix}
			f'(\phi_1)-b\phi_1 \\f'(\phi_2)-b\phi_2 \\ \cdots\\f'(\phi_{m^2}) -b\phi_{m^2}
		\end{matrix}
		\right]+\varepsilon^2\left(
		D_x^T\left[\begin{matrix}
			\mathcal{A}_1(\bm{p}_1)\\
			\mathcal{A}_1(\bm{p}_2)\\
			\cdots\\
			\mathcal{A}_1(\bm{p}_{m^2})
		\end{matrix}\right]+
		D_y^T\left[\begin{matrix}
			\mathcal{A}_2(\bm{p}_1)\\
			\mathcal{A}_2(\bm{p}_2)\\
			\cdots\\
			\mathcal{A}_2(\bm{p}_{m^2})
		\end{matrix}\right]
		\right).
	\end{aligned}
\end{equation}

\section{Proof of \cref{lemma_rearrange}}\label{appe_rearrange}
Let $ \bm{x}=(x_1,x_2) $. According to the Young's inequality, for any positive integers $m$ and $0\leqslant r\leqslant m$, it follows that
\[
\begin{aligned}
	|x_1^rx_2^{m-r}|\leqslant\frac{r}{m}|x_1|^m+\frac{m-r}{m}|x_2|^m\leqslant|\bm{x}|^{m},
\end{aligned}
\]
which indicates that for any given homogeneous polynomial $\xi_m$ of degree $ m $, there exists a constant $ C_0 $  such that
\begin{equation}\label{rerange}
	|\xi_m(\bm{x})|\leqslant C_0|\bm{x}|^{m}.
\end{equation}

Direct calculations lead to
\begin{equation*}
\frac{\partial\zeta_{_l}}{\partial x_i}=
\begin{cases}
\frac{\partial_{x_i} \eta_{_{l+2}}}{|\bm{x}|^l}-\frac{lx_i\eta_{_{l+2}}}{|\bm{x}|^{l+2}},& x=\bm{0},\\
0,& x\neq\bm{0}.
\end{cases}
\end{equation*}
Because $ \partial_{x_i} \eta_{_{l+2}} $ and $ lx_i\eta_{_{l+2}} $ are homogeneous polynomials of degree $l+1$ and $l+3$ respectively, it follows from \eqref{rerange} that $\lim_{\bm{x}\to\bm{0}}\frac{\partial\zeta_{_l}}{\partial x_i}=0 $ which implies $ \frac{\partial\zeta_{_l}}{\partial x_i} $ is continuous.

Similarly, for any  $ \bm{x}\in\mathbb{R}^2\backslash\{\bm{0}\} $ and $ i,j=1,2$, we have
\[
\frac{\partial^2\zeta_{_l}}{\partial x_i\partial x_j}=\frac{\partial_ {x_ix_j}\eta_{_{l+2}}}{|\bm{x}|^l}-\frac{lx_i\partial_ {x_j}\eta_{_{l+2}}+lx_j\partial_ {x_i}\eta_{_{l+2}}}{|\bm{x}|^{l+2}}+\frac{l(l+2)x_ix_j\eta_{_{l+2}}}{|\bm{x}|^{l+4}}.
\]
The same argument using \eqref{rerange} indicates that $ \frac{\partial^2\zeta_{_l}(\bm{x})}{\partial x_i\partial x_j} $ is bounded.

\section*{Acknowledgments}
The authors gratefully acknowledge many helpful discussions with Jin Zhang (Southern University of Science and Technology) and Chenglong Bao (Tsinghua University) during the preparation of the paper.

\bibliographystyle{siamplain}
\bibliography{refs}

\begin{thebibliography}{10}

\bibitem{an2005dc}
{\sc L.~T.~H. An and P.~D. Tao}, {\em The {DC} (difference of convex functions)
  programming and {DCA} revisited with {DC} models of real world nonconvex
  optimization problems}, Ann. Ope. Res., 133 (2005), pp.~23--46.

\bibitem{andreani2008augmented}
{\sc R.~Andreani, E.~G. Birgin, J.~M. Mart{\'\i}nez, and M.~L. Schuverdt}, {\em
  Augmented lagrangian methods under the constant positive linear dependence
  constraint qualification}, Math. Program., 111 (2008), pp.~5--32.

\bibitem{armelao2006}
{\sc L.~Armelao, D.~Barreca, G.~Bottaro, A.~Gasparotto, S.~Gross, C.~Maragno,
  and E.~Tondello}, {\em Recent trends on nanocomposites based on {Cu, Ag and
  Au clusters}: A closer look}, Coord. Chem. Rev., 250 (2006), pp.~1294--1314.

\bibitem{Jiang17}
{\sc W.~Bao, W.~Jiang, D.~J. Srolovitz, and Y.~Wang}, {\em Stable equilibria of
  anisotropic particles on substrates: a generalized {Winterbottom}
  construction}, SIAM J. Appl. Math., 77 (2017), pp.~2093--2118.

\bibitem{beck2017first}
{\sc A.~Beck}, {\em First-order methods in optimization}, SIAM, 2017.

\bibitem{bian2021three}
{\sc F.~Bian and X.~Zhang}, {\em A three-operator splitting algorithm for
  nonconvex sparsity regularization}, SIAM J. Sci. Comput., 43 (2021),
  pp.~A2809--A2839.

\bibitem{bolte2014proximal}
{\sc J.~Bolte, S.~Sabach, and M.~Teboulle}, {\em Proximal alternating
  linearized minimization for nonconvex and nonsmooth problems}, Math.
  Program., 146 (2014), pp.~459--494.

\bibitem{chen2019fast}
{\sc C.~Chen and X.~Yang}, {\em Fast, provably unconditionally energy stable,
  and second-order accurate algorithms for the anisotropic {Cahn--Hilliard}
  model}, Comput. Method. Appl. M., 351 (2019), pp.~35--59.

\bibitem{chen2013efficient}
{\sc F.~Chen and J.~Shen}, {\em Efficient energy stable schemes with spectral
  discretization in space for anisotropic {Cahn-Hilliard} systems}, Commun.
  Comput. Phys., 13 (2013), pp.~1189--1208.

\bibitem{cheng2020weakly}
{\sc K.~Cheng, C.~Wang, and S.~Wise}, {\em A weakly nonlinear, energy stable
  scheme for the strongly anisotropic {Cahn-Hilliard} equation and its
  convergence analysis}, J. Comput. Phys., 405 (2020), p.~109109.

\bibitem{clarke1990optimization}
{\sc F.~H. Clarke}, {\em Optimization and nonsmooth analysis}, SIAM, 1990.

\bibitem{danielson2006}
{\sc D.~Danielson, D.~Sparacin, J.~Michel, and L.~Kimerling}, {\em
  Surface-energy-driven dewetting theory of silicon-on-insulator
  agglomeration}, J. Appl. Phys., 100 (2006), p.~530.

\bibitem{davis2017three}
{\sc D.~Davis and W.~Yin}, {\em A three-operator splitting scheme and its
  optimization applications}, Set-valued Var. Anal., 25 (2017), pp.~829--858.

\bibitem{deckelnick2005computation}
{\sc K.~Deckelnick, G.~Dziuk, and C.~M. Elliott}, {\em Computation of geometric
  partial differential equations and mean curvature flow}, Acta Numer., 14
  (2005), pp.~139--232.

\bibitem{eyre1998unconditionally}
{\sc D.~J. Eyre}, {\em Unconditionally gradient stable time marching the
  {Cahn-Hilliard} equation}, MRS Online Proceedings Library (OPL), 529 (1998).

\bibitem{Fonseca91}
{\sc I.~Fonseca and S.~M{\"u}ller}, {\em A uniqueness proof for the {Wulff}
  theorem}, Proc. Royal Soc. Edinburgh, 119 (1991), pp.~125--136.

\bibitem{Gibbs1878}
{\sc J.~W. Gibbs}, {\em On the equilibrium of heterogeneous substances}, Trans.
  Connecticut Acad. Arts Sci., 3 (1878), pp.~104--248.

\bibitem{gunzburger1995perspectives}
{\sc M.~D. Gunzburger}, {\em Perspectives in flow control and optimization},
  SIAM, 2002.

\bibitem{jiang2019sharp}
{\sc W.~Jiang and Q.~Zhao}, {\em Sharp-interface approach for simulating
  solid-state dewetting in two dimensions: A {Cahn--Hoffman} $\xi$-vector
  formulation}, Phys. D, 390 (2019), pp.~69--83.

\bibitem{kobayashi1993modeling}
{\sc R.~Kobayashi}, {\em Modeling and numerical simulations of dendritic
  crystal growth}, Phys. D, 63 (1993), pp.~410--423.

\bibitem{li2016douglas}
{\sc G.~Li and T.~K. Pong}, {\em Douglas--rachford splitting for nonconvex
  optimization with application to nonconvex feasibility problems}, Math.
  Program., 159 (2016), pp.~371--401.

\bibitem{makki2016existence}
{\sc A.~Makki and A.~Miranville}, {\em Existence of solutions for anisotropic
  {Cahn-Hilliard} and {Allen-Cahn} systems in higher space dimensions},
  Discrete Cont. Dyn-S, 9 (2016), p.~759.

\bibitem{shen2018stabilized}
{\sc J.~Shen and J.~Xu}, {\em Stabilized predictor-corrector schemes for
  gradient flows with strong anisotropic free energy}, Commun. Comput. Phys.,
  (2018).

\bibitem{shen2010numerical}
{\sc J.~Shen and X.~Yang}, {\em Numerical approximations of {Allen-Cahn} and
  {Cahn-Hilliard} equations}, Discrete Cont. Dyn-A, 28 (2010), p.~1669.

\bibitem{sun2015convergent}
{\sc D.~Sun, K.-C. Toh, and L.~Yang}, {\em A convergent 3-block semiproximal
  alternating direction method of multipliers for conic programming with 4-type
  constraints}, SIAM J. Optim., 25 (2015), pp.~882--915.

\bibitem{thompson2012}
{\sc C.~Thompson}, {\em Solid-state dewetting of thin films}, Annu. Rev. Mater.
  Res., 42 (2012), pp.~399--434.

\bibitem{torabi2009new}
{\sc S.~Torabi, J.~Lowengrub, A.~Voigt, and S.~Wise}, {\em A new phase-field
  model for strongly anisotropic systems}, Proc. R. Soc. A, 465 (2009),
  pp.~1337--1359.

\bibitem{wise2007solving}
{\sc S.~Wise, J.~Kim, and J.~Lowengrub}, {\em Solving the regularized, strongly
  anisotropic {Cahn--Hilliard} equation by an adaptive nonlinear multigrid
  method}, J. Comput. Phys., 226 (2007), pp.~414--446.

\bibitem{wulff1901}
{\sc G.~Wulff}, {\em Zur frage der geschwindigkeit des wachstums und
  deraufl{\"o}sung der krystallfl{\"a}chen}, Z. Kristallogr, 34 (1901),
  pp.~449--530.

\bibitem{zhao2020parametric}
{\sc Q.~Zhao, W.~Jiang, and W.~Bao}, {\em A parametric finite element method
  for solid-state dewetting problems in three dimensions}, SIAM J. Sci.
  Comput., 42 (2020), pp.~B327--B352.

\bibitem{zhao2021energy}
{\sc Q.~Zhao, W.~Jiang, and W.~Bao}, {\em An energy-stable parametric finite
  element method for simulating solid-state dewetting}, IMA J. Numer. Anal., 41
  (2021), pp.~2026--2055.

\end{thebibliography}
\end{document}